\documentclass[a4paper,10pt]{amsart}

\pdfoutput=1

\usepackage[T1]{fontenc}
\usepackage[utf8]{inputenc}
\usepackage{newunicodechar}

\usepackage[usenames,dvipsnames,svgnames,x11names,table]{xcolor}

\usepackage[]{accents,
amsmath,
amsfonts,
amssymb,
amsthm,
amscd, 
amsxtra,
amstext,
latexsym,
syntonly,
tabularx}

\usepackage{pgf,tikz}
\usetikzlibrary{arrows}

\usepackage{enumitem}
\usepackage{graphicx}

\usepackage{mathrsfs}
\usepackage{stmaryrd}
\usepackage[normalem]{ulem}

\usepackage[all]{xy}

\usepackage[pagebackref, colorlinks=true, linkcolor=
Firebrick4,
citecolor=
Turquoise4,
filecolor=SteelBlue4,
urlcolor=
SteelBlue4
]{hyperref}

\theoremstyle{plain}
\newtheorem{xx}{xx}[section]
\newtheorem{thm}[xx]{Theorem}%[section]
\newtheorem*{thm*}{Theorem}
\newtheorem*{thmA*}{Theorem A}
\newtheorem*{thmB*}{Theorem B}
\newtheorem{prop}[xx]{Proposition}%[theorem]
\newtheorem{cor}[xx]{Corollary}%[theorem]
\newtheorem{lem}[xx]{Lemma}%[theorem]
\newtheorem*{lem*}{Lemma}%[theorem]

\theoremstyle{definition}
\newtheorem{defn}[xx]{Definition}%[theorem]
\newtheorem*{defn*}{Definition}%[theorem]
\newtheorem{ex}[xx]{Example}%[theorem]
\theoremstyle{remark}
\newtheorem{rem}[xx]{Remark}
\newtheorem*{rem*}{Remark}

\numberwithin{equation}{xx}
\DeclareMathOperator{\Bl}{Bl}
\DeclareMathOperator{\ch}{c}
\DeclareMathOperator{\Char}{char}

\DeclareMathOperator{\codim}{codim}
\DeclareMathOperator{\coker}{coker}
\DeclareMathOperator{\Def}{Def}

\DeclareMathOperator{\cDef}{\mathsf{Def}}

\DeclareMathOperator{\depth}{depth}

\DeclareMathOperator{\cE}{E}

\DeclareMathOperator{\End}{End}

\DeclareMathOperator{\ShEnd}{\ms{E}\!\mathit{nd}}

\DeclareMathOperator{\Ext}{Ext}

\DeclareMathOperator{\ShExt}{\ms{E}\hspace{-0.09em}\mathnormal{xt}}

\DeclareMathOperator{\Grass}{Grass}
\DeclareMathOperator{\cH}{H}

\DeclareMathOperator{\Hom}{Hom}

\DeclareMathOperator{\sHom}{\ms{H}\hspace{-0.15em}\mathnormal{om}}

\DeclareMathOperator{\id}{{id}}
\DeclareMathOperator{\im}{im}

\DeclareMathOperator{\Pic}{Pic}

\DeclareMathOperator{\Proj}{Proj}

\DeclareMathOperator{\Quot}{Quot}
\DeclareMathOperator{\Res}{Res}

\DeclareMathOperator{\rk}{rk}
\DeclareMathOperator{\Sets}{\cat{Sets}}

\DeclareMathOperator{\Sl}{SL}

\DeclareMathOperator{\Spec}{Spec}

\DeclareMathOperator{\Supp}{Supp}

\DeclareMathOperator{\Syz}{Syz}

\newcommand{\co}{\colon}
\newcommand{\ra}{\rightarrow}
\newcommand{\la}{\leftarrow}
\newcommand{\lra}{\longrightarrow}

\newcommand{\thr}{\twoheadrightarrow}

\newcommand{\hra}{\hookrightarrow}

\newcommand{\llb}{\llbracket}
\newcommand{\rrb}{\rrbracket}
\newcommand{\ot}{\hspace{0.06em}{\otimes}\hspace{0.06em}}

\newcommand{\lRa}{\Leftrightarrow}

\newcommand{\Ra}{\Rightarrow}

\newcommand{\sbeq}{\subseteq}

\newcommand{\vare}{\varepsilon}

\newcommand{\vL}{\varLambda}

\newcommand{\dprime}{\prime\hspace{-0.1em}\prime}

\newcommand{\tri}{\vartriangle}

\newcommand\bdot{\ensuremath{%
  \mathchoice%
   {\mskip\thinmuskip\lower0.2ex\hbox{\scalebox{1.5}{$\cdot$}}\mskip\thinmuskip}}%
   {\mskip\thinmuskip\lower0.2ex\hbox{\scalebox{1.5}{$\cdot$}}\mskip\thinmuskip}%        
   {\lower0.3ex\hbox{\scalebox{1.2}{$\cdot$}}}%  
   {\lower0.3ex\hbox{\scalebox{1.2}{$\cdot$}}}%
}

\newcommand{\ccirc}{\mathbin{\mathchoice
  {\xcirc\scriptstyle}
  {\xcirc\scriptstyle}
  {\xcirc\scriptscriptstyle}
  {\xcirc\scriptscriptstyle}
}}
\newcommand{\xcirc}[1]{\vcenter{\hbox{$#1\circ$}}}

\renewcommand{\phi}{\varphi}
\renewcommand{\geq}{\geqslant}
\renewcommand{\leq}{\leqslant}

\newcommand{\BB}[1]{\mathbb{{#1}}}

\newcommand{\bs}[1]{\boldsymbol{{#1}}}

\newcommand{\fr}[1]{\mathfrak{{#1}}}

\newcommand{\cat}[1]{\mathsf{{#1}}}

\newcommand{\mr}[1]{\mathrm{{#1}}}
\newcommand{\ms}[1]{\mathscr{{#1}}}
\newcommand{\tn}{\textnormal}

\newcommand{\df}[2]{{\Def}_{\hspace{-0.08em}#2}^{#1}}
\newcommand{\cdf}[2]{{\cDef\hspace{-0.08em}}_{#2}^{#1}}

\newcommand{\hm}[4]{{\Hom}_{#2}^{#1}({#3},{#4})}

\newcommand{\nd}[3]{{\End} _{#2}^{#1}({#3})}

\newcommand{\Q}{\hspace{-0.05em}\mathcal{O}\hspace{-0.05em}}

\newcommand{\shm}[4]{{\sHom} _{#2}^{#1}({#3},{#4})}

\newcommand{\snd}[3]{{\ShEnd}_{#2}^{#1}({#3})}

\newcommand{\sxt}[4]{{\ShExt}_{#2}^{#1}({#3},{#4})}

\newcommand{\xla}[1]{\xleftarrow{{#1}}}
\newcommand{\xra}[1]{\xrightarrow{{#1}}}
\newcommand{\xt}[4]{{\Ext}_{#2}^{#1}({#3},{#4})}

\newcommand{\syz}[2]{{\Syz}_{#2}^{#1}}
\begin{document}
\title[Reflexive modules on rational surface singularities and flops]
{Deformations of rational surface singularities and reflexive modules\\ with an application to flops
}
\author{Trond St{\o}len Gustavsen and Runar Ile}
\address{University of Southeast Norway/Department of Mathematics \\University of Bergen \\Norway}
\email{trond.gustavsen@math.uib.no}
\address{BI Norwegian Business School/Department of Mathematics \\University of Bergen \\Norway}
\email{runar.ile@bi.no}
\keywords{Flatifying blowing-up, maximal Cohen-Macaulay module, simultaneous partial resolution, small resolution, rational double point, matrix factorisation} 
\subjclass[2010]
{Primary 14B07, 14E30; Secondary 14D23, 14E16}
\begin{abstract}
Blowing up a rational surface singularity in a reflexive module gives a (any) partial resolution dominated by the minimal resolution. 
The main theorem shows how deformations of the pair (singularity, module) relates to deformations of the corresponding pair of partial resolution and locally free strict transform, and to deformations of the underlying spaces. The results imply some recent conjectures on small resolutions and flops.
\end{abstract}
\maketitle
%%%%%%%%%%%%%%%%%%%%%%%%%%%%%%%%%%%%%%%%%%%%%%%%%%%%%%
%%%%%%%%%%%%%% SECTION %%%%%%%%%%%%%%%%%%%%%%%%%%%%%%%
%%%%%%%%%%%%%%%%%%%%%%%%%%%%%%%%%%%%%%%%%%%%%%%%%%%%%%
\section{Introduction}
We relate deformations of a rational surface singularity with a reflexive module to deformations of a partial resolution of the singularity with the locally free strict transform of the module.
Our results imply three conjectures of C. Curto and D. Morrison about how a family of small resolutions of a \(3\)-dimensional index one terminal singularity and its flop are obtained by blowing up in a maximal Cohen-Macaulay module and its syzygy. 

Rational surface singularities were defined by M.\ Artin in \cite{art:66}. Further foundational work was done by E.\ Brieskorn \cite{bri:68} and J.\ Lipman \cite{lip:69} and many studies have followed. In the 1980s the geometrical McKay correspondence was establised by G. Gonzales-Sprinberg and J.-L. Verdier \cite{gon-spr/ver:83} and generalised in \cite{art/ver:85}.
It gives a bijection between the isomorphism classes of (non-projective) indecomposable reflexive modules \(\{M_{i}\}\) and the prime components \(\{E_{j}\}\) of the exceptional divisor in the minimal resolution \(\tilde{X}\ra X\) of a rational double point (RDP), i.e. the \(\mr{A}_{n}\), \(\mr{D}_{n}\) and \(\mr{E}_{6-8}\). More precisely, if \(\ms{F}_{i}\) denotes the strict transform of \(M_{i}\) to \(\tilde{X}\), the Chern class of \(\ms{F}_{i}\) 
is dual to the prime divisor; \(\ch_{1}(\ms{F}_{i}).E_{j}=\delta_{ij}\), with \(\rk M_{i}\) equal to the multiplicity of \(E_{i}\) in the fundamental cycle.
For non-Gorenstein quotient surface singularities there are in general more indecomposable reflexive modules than prime components as was shown by H.\ Esnault \cite{esn:85}. However, O.\ Riemenschneider and his student J.\ Wunram gave a natural class of `special' reflexive modules (which we will call Wunram modules) for which the correspondence holds for any rational surface singularity \cite{rie:87,wun:88}. A.\ Ishii refined Wunram's result by means of a Fourier-Mukai transform in the case of quotient surface singularities \cite{ish:02}. M. Van den Bergh's use in \cite{vdber:04} of the endomorphism ring of a higher dimensional Wunram module to prove derived equivalences for flops induced a lot of activity, also attracting attention to the \(2\)-dimensional case with interesting results by M. Wemyss and collaborators, e.g.\ O.\ Iyama and Wemyss \cite{iya/wem:10,iya/wem:11} and Wemyss \cite{wem:11}.

The McKay-Wunram correspondence is foundational for this article: We prove that blowing up a rational surface singularity \(X\) in a reflexive module \(M\) (a special case of L. Gruson and M. Raynaud's flatifying blowing-up \cite{ray/gru:71}) gives a partial resolution \(f\co Y\ra X\) where \(Y\) in particular is normal, dominated by the minimal resolution, and the strict transform \(\ms{M}=f^{\tri}(M)\) is locally free. The partial resolution is determined by the Chern class \(\ch_{1}(\ms{F})\) of the strict transform $\ms{F}$ of $M$ to $\tilde{X}$. In particular, any partial resolution dominated by the minimal resolution is given by blowing up in a Wunram module. See Theorem \ref{thm.MW} for more precise statements. As an example, the RDP-resolution (obtained by contracting the \((-2)\)-curves in the minimal resolution) is given by blowing up in the canonical module \(\omega_{X}\).

Consider the deformations \(\df{}{(Y,\ms{M})}\) of the pair \((Y,\ms{M})\) which blow down to deformations of the pair \((X,M)\). Our main result (Theorem \ref{thm.square}) says that in the commutative diagram of deformation functors
\begin{equation*}
\xymatrix@C+12pt@R-6pt@H-6pt{
\hspace{-3.0em}\df{}{(Y,\ms{M})} \ar[r]^(0.45){\beta}\ar@<-2.5em>[d]_(0.45){\alpha} & \df{}{Y}\hspace{-0.5em}\ar@<0.12em>[d]^(0.45){\delta}
\\
\hspace{-3em}\df{}{(X,M)} \ar[r] & \df{}{X}\hspace{-0.5em}
}
\end{equation*}
the blowing down map \(\alpha\) is injective and the forgetful map \(\beta\) is smooth and in many situations an isomorphism. The injectivity of \(\alpha\) is surprising since the blowing down map \(\delta\) in general is not injective (cf. Remark \ref{rem.A} and \cite[6.4]{wah:75}). On spaces \(\delta\) is a Galois covering onto the Artin component \(A\) which for RDPs equals \(\df{}{X}\) \cite{bri:70,tju:70,pin:83,art:74b,wah:79}.
However, \(\beta\) is an isomorphism if \(M\) is Wunram (e.g. any reflexive on an RDP) implying that \(\delta\) factors through a closed embedding \(\alpha\beta^{-1}\co\df{}{Y}\sbeq \df{}{(X,M)}\) realising deformations of the partial resolution as deformations of the pair as conjectured by Curto and Morrison in the RDP case. A deformation of \(X\) in the component \(A\) lifts in general to a deformation of \((X,M)\) -- and of \(Y\) -- only after a finite base change. However, a deformation of the pair \((X,M)\) in the geometric image of \(\df{}{(Y,\ms{M})}\) lifts to a deformation of \((Y,\ms{M})\) without any base change. Note that \(\df{}{(X,M)}\) in general is not dominated by \(\df{}{(Y,\ms{M})}\), even for RDPs: in Example \ref{ex.fund}, \(M\) is the (rank two) fundamental module and \(\df{}{(X,M)}\) has two components while \(\df{}{Y}\) has one. A crucial ingredient (first proved by Lipman \cite{lip:79}) in J.\ Wahl's proof that the covering \(\df{}{{\tilde{X}}}\ra A\) has Galois action by a product of Weyl groups was the injectivity of \(\delta\) in the case \(Y\) is the RDP-resolution. This is an immediate consequence of our main result since \(\df{}{(X,\,\omega_{X})}\cong \df{}{X}\); see Corollary \ref{cor.Lip}. 
While knowledge of \(\df{}{(X,M)}\) would be interesting in itself, these results also indicate that there are interesting relations to \(\df{}{X}\), e.g. regarding the component structure.

In this article our main application of Theorem \ref{thm.square} is a generalisation of three conjectures of Curto and Morrison \cite{cur/mor:13} concerning the nature of small partial resolutions of \(3\)-dimensional index one terminal singularities and their flops. If \(g\co W\ra Z\) is such a small partial resolution and \(X\sbeq Z\) is a sufficiently generic hyperplane section with strict transform \(f\co Y\ra X\), a result of M. Reid \cite{rei:83} says that \(f\) is a partial resolution (normal, dominated by the minimal resolution) of an RDP. In particular, \(g\) is a \(1\)-parameter deformation of \(f\) and hence an element in \(\df{}{Y}\). By Theorem \ref{thm.MW}, \(Y\) is the blowing-up of \(X\) in a reflexive module \(M\). Then \(\alpha\beta^{-1}\) takes \(g\) to a \(1\)-parameter deformation \((Z,N)\) of the pair \((X,M)\). The basic result is the following (cf. Theorem \ref{thm.flop}):
%%%%%%%%%%%%%%%% TEOREM %%%%%%%%%%%%%%%%%%%%%% 
\begin{cor}\label{cor.A}
There is a maximal Cohen-Macaulay \(\Q_{Z}\)-module \(N\) such that\textup{:}
\begin{enumerate}[leftmargin=2.4em, label=\textup{(\roman*)}] 
\item The small partial resolution \(W\ra Z\) is given by blowing up \(Z\) in \(N\)\textup{.} 
\item Blowing up \(Z\) in the syzygy module \(N^{+}\) of \(N\) gives the unique flop \(W^{+}\hspace{-0.3em}\ra Z\)\textup{.}
\item The length of the flop equals the rank of \(N\) if the flop is simple\textup{.}
\end{enumerate} 
\end{cor}
Theorem \ref{thm.CM} is a version of this statement for flat families of such small partial resolutions and flops. There is a family of pairs \((\bs{X},\bs{M})\) in \(\df{}{(X,M)}\) such that the blowing up of \(\bs{X}\) in \(\bs{M}\) and in the syzygy \(\bs{M}^{+}\) give two simultaneous partial resolutions \(\bs{Y}\ra\bs{X}\la\bs{Y}^{+}\) which induce any local family of flops of \(g\) by pullback, for any $g$ with hyperplane section $f$. By a result of S. Katz and Morrison, in the simple case the length \(l\) of the flop determines the generic hyperplane section \(X\) \cite{kat/mor:92}, see also \cite{kaw:94}. More precisely, \(X\) equals \(\mr{A}_{1},\mr{D}_{4},\mr{E}_{6},\mr{E}_{7},\mr{E}_{8}\) or \(\mr{E}_{8}\) for \(l=1,2,3,4,5\) or \(6\), respectively. By our result there is in each case a unique reflexive module \(M\) of rank \(l\) such that any simple flop of length \(l\) is obtained by pullback from the \(\bs{Y}\ra\bs{X}\la\bs{Y}^{+}\)
for the corresponding \((\bs{X},\bs{M})\). Hence \(\bs{Y}\ra\bs{X}\la\bs{Y}^{+}\) gives the `universal' simple flop of length \(l\) realised as blowing-ups in families of reflexive modules as suggested by Curto and Morrison; see Remark \ref{rem.uni}.

As an example consider \(\mr{A}_{1}\co x^{2}+yz\) which has a minimal versal family \(x^{2}+yz-u\). After the base change \(u\mapsto t^{2}\) it allows a simultaneous deformation of the minimal resolution and the resulting family is a small resolution of \(Z\co x^{2}+yz-t^{2}\) with exceptional fibre \(E\cong \BB{P}^{1}\); see M. F. Atiyah \cite[Thm. 2]{ati:58}. The only non-trivial indecomposable reflexive module \(M\) on \(\mr{A}_{1}\) extends to a module \(N\) on \(Z\) with presentation matrix 
\(\Phi=
\left(
\begin{smallmatrix}
x+t & y \\
-z & x-t
\end{smallmatrix}
\right)
\).
Blowing up \(Z\) in \(N\) gives the simultaneous resolution \(W\ra Z\) of the family. Blowing up \(Z\) in the syzygy \(N^{+}\hspace{-0.2em}\) gives the simple flop \(W^{+}\hspace{-0.3em}\ra Z\) of length one. The presentation matrix of \(N^{+}\) is the adjoint \(\Psi\) of \(\Phi\) and the pair makes a matrix factorisation of the hypersurface \(Z\). The RDPs are hypersurfaces and any maximal Cohen-Macaulay module is given by a matrix factorisation \cite{eis:80}. Curto and Morrison phrase their conjectures in terms of matrix factorisations (and for simple flops) and verify them for the \(\mr{A}_{n}\) and \(\mr{D}_{n}\) by extensive calculations. The higher ranks of the indecomposable modules for the \(E_{6-8}\) makes this approach difficult, and for the non-simple flops practically impossible. Our argument is conceptual and does not rely on computations. The coordinate-free formulation of Theorems \ref{thm.flop} and \ref{thm.CM} makes the conjectures more transparent and accessible; see Remark \ref{rem.CM}. By a result of O. Villamayor U. generators for the blowing-up ideal are readily obtained from a presentation of the module \cite{vil:06}, cf. comments below \eqref{eq.OZ}. The singularities we work with are henselisations of finite type algebras and the results will therefore have finite type representations locally in the \'etale topology.

In recent years there has been a lot of research linking properties of various noncommutative algebras and the flops, e.g. notably the description by W.\ Donovan and Wemyss of the Bridgeland-Chen autoequivalence in terms of the universal family of a non-commutative deformation functor \cite{don/wem:16}. 
J. Karmazyn \cite{kar:17} reconstructs the small partial resolution and its flop by a quiver GIT-construction where the input is endomorphism algebras. Wemyss \cite{wem:18} contains many general results describing flops and minimal models of singularities (e.g. for cDVs) in homological terms. In particular he describes flops in terms of mutations, with applications to the GIT chamber structure.
We offer on the other hand a direct proof of the original Curto-Morrison conjectures using deformation theory where the blowing-up ideal for the small, partial resolution is obtained directly from the (parametrised) \(2\)-dimensional Wunram module.
Moreover, any flop with fixed RDP hyperplane section and Dynkin diagram is a pullback from a pair of such `universal' blowing-ups. We also believe that the geometric techniques used in this article may be useful in the study of more general contractions. See Remarks \ref{rem.W} and \ref{rem.VdB}.

The inventory of the article is as follows. 
In Section \ref{sec.prel} we give preliminary results concerning rational surface singularities, blowing-up in coherent sheaves, strict transforms on partial resolutions and their Chern classes and a cohomology and base change result suited to our needs. 
In Section \ref{sec.def} we define the deformation functors. We also give a result which implies the compatibility of blowing-up in a family of modules with base change. 
In Section \ref{sec.MW} we prove a result concerning the fractional ideal which defines the blowing-up,  normality of blowing-up, and the blowing-up version of the McKay-Wunram correspondence. 
In Section \ref{sec.square} we prove the main theorem through several intermediate steps. Existence of versal base spaces and a classical result of Lipman follows. There is also an example (the fundamental module). The article ends in Section \ref{sec.flops} with our treatment of the Curto-Morrison conjectures.
\subsection*{Acknowledgement}
Part of this work was done during the first author's most pleasant stay at Northeastern University 2013/14. 

The authors thank the referee for a detailed and helpful report.
%%%%%%%%%%%%%%%%%%%%%%%%%%%%%%%%%%%%%%%%%%%%%%%%%%%%%%
%%%%%%%%%%%%%% SECTION %%%%%%%%%%%%%%%%%%%%%%%%%%%%%%%
%%%%%%%%%%%%%%%%%%%%%%%%%%%%%%%%%%%%%%%%%%%%%%%%%%%%%%
\section{Preliminaries}\label{sec.prel}
%%%%%%%%%%%%%%%%%%%%%%%%%%%%%%%%%%%%%%%%%%%%%%%%%%%%%% 
\subsection{Partial resolutions of rational surface singularities}
Fix an algebraically closed field \(k\). All schemes and maps are assumed to be above \(\Spec k\) and all schemes are assumed to be noetherian.
%%%%%%%%%%%%%%%%%%%%%%%%%%%%%%%%%%%%%%%%%%%%%%%%%%%%%%%%%%%%
\begin{defn}
A \emph{singularity} is an affine scheme \(X=\Spec A\) where \(A\) is algebraic (the henselisation of a finite type \(k\)-algebra in a maximal ideal). A \emph{partial resolution} of \(X\) is a proper birational map \(f\co Y\ra X\) with \(Y\) normal. If \(Y\) is regular, \(f\) is a resolution. Let \(E(f)\subset Y\) denote the (non-reduced) closed fibre of \(f\) and let \(\varSigma(f)\) denote the exceptional set of \(f\); the minimal closed subset of \(Y\) such that \(f\) restricted to its complement is an isomorphism. A partial resolution \(f\) is \emph{small} if \(\varSigma(f)\) does not contain any divisorial components.
 
If \(A\) furthermore is a normal domain of dimension two, \(X\) is called a \emph{normal surface singularity}. Moreover, \(X\) is a \emph{rational} surface singularity if there is a resolution \(f\) such that \(\mr{R}^{1}\hspace{-0.2em}f_{*}\Q_{Y}=0\); \cite{art:66}. A rational surface singularity which is a double point is called a \emph{rational double point} (RDP). 
\end{defn}
A normal surface singularity is an RDP if and only if it is a Gorenstein rational surface singularity; cf.\ \cite[4.19]{bad:01}.
RDP is also equivalent to Du Val as defined in \cite[4.4]{kol/mor:98}; cf.\ \cite[3.31, 4.1]{bad:01}. 
A finite module on a normal surface singularity is reflexive if and only if it is maximal Cohen-Macaulay (MCM). 

A fundamental reference for the following results is Lipman \cite{lip:69}. Proposition \ref{prop.modific} will be used without further mentioning.
%%%%%%%%%%%%%%%%%%% PROPOSITION %%%%%%%%%%%%%%%%%%%%%%%%%%%%
\begin{prop}[{\cite[4.1, 27.1]{lip:69}}]\label{prop.modific}
Let \(X\) be a rational surface singularity and \(f\co Y\ra X\) a partial resolution\textup{.} Let \(\{E_{i}\}_{i\in I}\) denote the prime components of \(E(f)\)\textup{.} There is a \emph{minimal resolution of singularities} \(\pi\co \tilde{X}\ra X\) \textup{(}independent of \(f\)\textup{)} such that\textup{:}
\begin{enumerate}[leftmargin=2.4em, label=\textup{(\roman*)}, start=1]
\item\textup{(Minimality)} If \(f\) is a resolution of singularities then there exists a unique map \(g\co Y\ra \tilde{X}\) such that \(f=\pi g\)\textup{.} 
\item\textup{(Singularities)} \(Y\) has only rational surface singularities\textup{.} If \(X\) is an RDP then \(Y\) has only RDP singularities\textup{.}
\item \textup{(Contracting exceptional curves)} 
For any subset \(J\sbeq I\) there exists a unique partial resolution \(g\co Y_{\hspace{-0.09em}J}\ra X\) and map \(h\co Y\ra Y_{\hspace{-0.09em}J}\) with \(f=gh\) such that \(g\) contracts exactly the curves \(\{E_{i}\}_{i\in I\setminus J}\)\textup{.} 
\end{enumerate}
\end{prop}
\begin{proof}
For the minimal resolution, (i) and (ii) see \cite[4.1 and 1.2]{lip:69}.
For (iii) see 27.1 and \emph{Remarks} p. 275 in \cite{lip:69}.   
\end{proof}
%%%%%%%%%%%%%%%%%%% PROPOSITION %%%%%%%%%%%%%%%%%%%%%%%%%%%%
\begin{prop}\label{prop.pic}
Let \(X\) be a normal singularity of dimension at least two and suppose \(f\co Y\ra X\) is a partial resolution\textup{.}
Let \(\{E_{i}\}_{i\in I}\) denote the prime components of \(E(f)\)\textup{.} Assume \(\dim E_{i}=1\) for all \(i\in I\) and \(\mr{R}^{1}\hspace{-0.1em}f_{*}\Q_{Y}=0\)\textup{.} Then\textup{:} 
\begin{enumerate}[leftmargin=2.4em, label=\textup{(\roman*)}, start=1]
\item \({E}_{j}\cong\BB{P}^{1}\) for all \(j\), the intersections are transversal and \(E(f)\) contains no embedded components\textup{.}
\item \textup{(Intersection numbers)}
Let \(\ms{L}\) be an invertible sheaf on \(Y\) and \(C\in\{{E}_{i}\}_{i\in I}\)\textup{.} Put \(\ms{L}\hspace{-0.1em}.\hspace{0.05em}C=\mr{deg}_{C}(\ms{L}\ot\Q_{C})\)\textup{;} cf. \cite[\S 10-11]{lip:69}\textup{.} Then\textup{:}  
\begin{enumerate}[leftmargin=2.4em, label=\textup{(\alph*)}]
\item \(\ms{L}\cong\Q_{Y}\) if and only if \(\ms{L}\hspace{-0.1em}.\hspace{0.05em}C=0\) for all \(C\in\{{E}_{i}\}_{i\in I}\)\textup{.}
\item \(\ms{L}\) is generated by its global sections if and only if \(\ms{L}\hspace{-0.1em}.\hspace{0.05em}C\geq 0\) for all \(C\in\{{E}_{i}\}_{i\in I}\)\textup{.} In that case \(\mr{R}^{1}\hspace{-0.15em}f_{*}\ms{L}=0\)\textup{.}
\item \(\ms{L}\) is ample if and only if \(\ms{L}\hspace{-0.1em}.\hspace{0.05em}C>0\) for all \(C\in\{{E}_{i}\}_{i\in I}\)\textup{.} In that case \(\ms{L}\) is very ample for \(f\)\textup{.} 
\end{enumerate}
\item \textup{(The Picard group)} 
For each \(i\in I\) there is an effective prime Cartier divisor \({D}_{i}\) 
which intersects \(\cup_{i\in I}{E}_{i}\) transversally in a point contained in \({E}_{i}\)\textup{.} Moreover\textup{,} \(\{{D}_{i}\}_{i\in I}\) gives a \(\BB{Z}\)-basis for \(\Pic (Y)\)\textup{.}
\item \textup{(Hyperplane sections)}  
Assume \(f\) is small
and \(\dim X\geq 3\)\textup{.} 
Let \(g\co H'\ra H\) denote the strict transform along \(f\) of a hyperplane section \(H\subset X\) defined by a non-zero-divisor \(u\)\textup{.} Assume that \(H\) and \(H'\) are normal\textup{.} Then  
the restriction map \(\Pic(Y)\ra \Pic(H')\) is an isomorphism\textup{.}
Moreover\textup{,}
\begin{equation*}
\Q(D_{i}).E(f)=\Q(D_{i}\cap H').E(g)\textup{.}
\end{equation*} 
\end{enumerate}
\end{prop}
\begin{proof}
(i) Note that \(0=\mr{R}^{1}\hspace{-0.1em}f_{*}\Q_{Y}\thr\mr{R}^{1}\hspace{-0.1em}f_{*}\Q_{C}\)
for all subschemes \(C\) with support in \(\cup {E}_{j}\). It follows that \(p_{\tn{a}}({E}_{j})=0\) (which implies \({E}_{j}\cong\BB{P}^{1}\)) and that the intersections are transversal. Since \(f_{*}\Q_{Y}\thr f_{*}\Q_{E(f)}\) and \(f_{*}\Q_{Y}=\Q_{X}\) by \cite[\href{https://stacks.math.columbia.edu/tag/0AY8}{Lemma 0AY8}]{SP},   it follows that \(\cH^{0}(\Q_{E(f)})\cong k\) and \(E(f)\) cannot have embedded components.
(ii) is \cite[12.1]{lip:69}.

(iii) We imitate the proof of \cite[14.3]{lip:69}. Let \(y\in {E}_{i}\smallsetminus \cup_{j\neq  i}{E}_{j}\) be a closed point and \(\bar{t}\) a generator for the maximal ideal in \(\Q_{{E}_{i},y}\). Let 
\(t\in\Q_{Y,y}\) be a lifting of \(\bar{t}\).  
One may assume that no \(E_{j}\) is a component of the principal Cartier divisor \((t)\). Put \((t)=D_{i}+D'_{i}\) where \(D_{i}\cap(\cup E_{j})=\{y\}\) and \(y\notin D_{i}'\) (use that \(X\) is henselian).  
There is a map \(\theta\co\Pic(Y)\ra \hm{}{\BB{Z}}{\oplus_{i} \BB{Z}E_{i}}{\BB{Z}}\) given by \(\ms{L}\mapsto (\ms{L}.-)\). The existence of \(D_{i}\) shows surjectivity of \(\theta\) and (ii) shows injectivity.

(iv) Note that the strict transform equals the total transform. In particular, \(\{E_{i}\}_{i\in I}\) are the prime components of \(g^{-1}(x)\). The sequence \((u,t)\) is \(\Q_{Y,y}\)-regular. It implies that the standard Cartier divisor in \(\Pic (H')\) given in (iii) corresponding to the prime component \(E_{i}\) can be taken to be \(D_{i}\cap H'\). Since \(\Q_{E(f),y}\cong \Q_{E(g),y}\) the moreover part follows.
\end{proof}
\begin{rem} 
Note in (iii) that a Cartier divisor \(D\) which intersects \(\cup_{i\in I}{E}_{i}\)  transversally is contained in any open \(U\sbeq Y\) which containes the intersection points.
\end{rem}
%%%%%%%%%%%%%%%%%%%%%%%%%%%%%%%%%%%%%%%%%%%%%%%%%%%%%%
\subsection{Blowing up in coherent sheaves}
Let \(X\) be a scheme, \(i\co U\ra X\) a non-empty open subscheme with complement \(Z\), and \(\ms{F}\) a quasi-coherent \(\Q_{X}\)-module. 
Suppose \(f\co Y\ra X\) is a scheme map such that the restriction \(f_{U}\) of \(f\) to \(f^{-1}(U)\) is an isomorphism \(f^{-1}(U)\cong U\). Let \(j\co f^{-1}(U)\ra Y\) denote the open inclusion. Define the \emph{\(Z\)-strict transform of \(\ms{F}\) along \(f\)} to be the image of the natural restriction map \(f^{*}\ms{F}\ra j_{*}f_{U}^{*}(\ms{F}_{|U})\) -- a quasi-coherent \(\Q_{Y}\)-module  
denoted \(f^{\tri}_{Z}\ms{F}\). The kernel of the restriction map is the subsheaf \(\ms{H}^{0}_{f^{-1}Z}(f^{*}\ms{F})\) of sections with support in \(f^{-1}(Z)\). Let \(U'\sbeq X\) be another open subscheme with \(f^{-1}(U')\cong U'\) and suppose \(\ms{F}_{|U\cup U'}\) is locally free and both \(f^{-1}(U)\) and \(f^{-1}(U')\) are dense in \(Y\). Then \(f^{\tri}_{Z}\ms{F}\cong f^{\tri}_{Z'}\ms{F}\). We use the simplified notation \(f^{\tri}\ms{F}\) for the maximal such \(U\) and call it the strict transform. If $Y$ is integral then $f^{-1}Z$ does not contain the generic point of $Y$ and all local sections of $\ms{H}^{0}_{f^{-1}Z}(f^{*}\ms{F})$ are torsion. If $\ms{F}_{\vert U}$ is locally free (as in the applications below), then all torsion local sections in $f^{*}\ms{F}$ have support in $f^{-1}Z$ since a locally free sheaf has no torsion; i.e. $\ms{H}^{0}_{f^{-1}Z}(f^{*}\ms{F})=(f^{*}\ms{F})_\tn{tors}$.

The following is a special case of Gruson and Raynaud's theorem on flattening blowing-up (with the universal property); cf.\ \cite[5.2.2]{ray/gru:71}.
%%%%%%%%%%%%%%%%%%%%%%%PROPOSISJON%%%%%%%%%%%%%%%%%%%%%%%%%%%%%%%
\begin{prop}\label{prop.blowup}
Suppose \(X\) is a scheme\textup{,} \(U\) an open subscheme of \(X\) and \(\ms{F}\) a coherent \(\Q_{X}\)-module such that \(\ms{F}_{|U}\) is locally free\textup{.} Put \(Z=X\setminus U\)\textup{.} Then there is a projective scheme map \(f\co Y\ra X\)  
which is universal with respect to the following properties for a scheme map \(f'\co Y'\ra X\)\textup{.} 
\begin{enumerate}[leftmargin=2.4em, label=\textup{(\roman*)}]
\item The restriction \(f'_{U}\) is an isomorphism and \(f'^{-1}(U)\) is dense in \(Y'\)\textup{.}
\item The \(Z\)-strict transform \(f'^{\tri}_{Z}\ms{F}\) is locally free on \(Y'\)\textup{.} 
\end{enumerate}
\end{prop}
The proof realises \(Y\) as the scheme-theoretic closed image (so possibly with non-reduced structure; \cite[9.5]{EGAI}) of a map from \(U\) to the scheme of quotients \(\Quot_{\ms{F}/X/X}\); see \cite[\S 5.2]{ray/gru:71}.
Denote \(Y\) by \(\Bl_{Z,\ms{F}}(X)\).
Let \(U'\sbeq X\) be another open subscheme of \(X\) with \(\ms{F}_{|U'}\) locally free and such that both \(U\) and \(U'\) are dense in \(U\cup U'\). Put \(Z'=X\setminus U'\). Then \(\Bl_{Z'\hspace{-0.2em},\ms{F}}(X)\) equals \(\Bl_{Z,\ms{F}}(X)\). The simplified notation \(f\co\Bl_{\ms{F}}(X)\ra X\) is used if \(U\) is maximal with \(\ms{F}_{|U}\) locally free and \(f\) is called the blowing-up of \(X\) in \(\ms{F}\).  
Note that A.~Oneto and E.~Zatini \cite{one/zat:91} defined the blowing-up as the closure of the image of \(U\) with reduced structure. Many of their results extend to the non-reduced context.

As we shall consider base changes of blowing-ups, the following corollary will be useful.
%%%%%%%%%%%%%%%%%%%%%%% KOROLLAR %%%%%%%%%%%%%%%%%%%%%%%%%%%%%%%
\begin{cor}\label{cor.blowup}
Given a commutative diagram of scheme maps
\begin{equation*}
\xymatrix@C+6pt@R-8pt@H-6pt{
Y_{2}\ar[d]_(0.4){f_{2}}\ar[r]^{g} & Y_{1}\ar[d]^(0.4){f_{1}} \\
X_{2}\ar[r]^{p} & X_{1}
}
\end{equation*}
and an open subscheme \(U_{1}\sbeq X_{1}\)\textup{.} Put \(U_{2}=p^{-1}(U_{1})\) and \(Z_{i}=X_{i}\setminus U_{i}\)\textup{.} Assume that \(f_{i}\) is an isomorphism above \(U_{i}\) and that \(f_{i}^{-1}(U_{i})\) is dense in \(Y_{i}\) for \(i=1,2\)\textup{.}  
Suppose \(\ms{F}\) is a coherent \(\Q_{X_{1}}\)-module such that \(\ms{F}_{|U_{1}}\) and \((f_{1})_{Z_{1}}^{\tri}\ms{F}\) are locally free\textup{.}
\begin{enumerate}[leftmargin=2.4em, label=\textup{(\roman*)}]
\item The natural map \(g^{*}((f_{1})_{Z_{1}}^{\tri}\ms{F})\ra (f_{2})_{Z_{2}}^{\tri}(p^{*}\ms{F})\) is an isomorphism\textup{.}
\item If \(f_{1}\) equals \(\Bl_{Z_{1},\ms{F}}(X_{1})\ra X_{1}\) and \(Y_{2}=\Bl_{Z_{1},\ms{F}}(X_{1})\times X_{2}\)\textup{,} 
then \(Y_{2}\) is isomorphic to \(\Bl_{Z_{2},p^{*}\hspace{-0.12em}\ms{F}}(X_{2})\) over \(X_{2}\)\textup{.}
\end{enumerate}
\end{cor}
\begin{proof}
(i) There is a natural map 
\begin{equation}
g^{*}\ms{H}^{0}_{f_{1}^{-1}Z_{1}}(f_{1}^{*}\ms{F})\lra \ms{H}^{0}_{f_{2}^{-1}Z_{2}}((pf_{2})^{*}\ms{F})
\end{equation}
inducing a surjection \(\phi\co g^{*}((f_{1})_{Z_{1}}^{\tri}\ms{F})\ra (f_{2})_{Z_{2}}^{\tri}(p^{*}\ms{F})\). Since \(\phi\) restricted to the dense \((f_{1}g)^{-1}(U_{1})\) is an isomorphism and \(g^{*}((f_{1})_{Z_{1}}^{\tri}\ms{F})\) is locally free, \(\phi\) is an isomorphism.

(ii) By (i), \((f_{2})_{Z_{2}}^{\tri}(p^{*}\ms{F})\) is locally free. By the universal property in Proposition \ref{prop.blowup} there is an \(X_{2}\)-map \(r\co Y_{2}\ra \Bl_{Z_{2},p^{*}\hspace{-0.12em}\ms{F}}(X_{2})\). Similarly, there is an \(X_{1}\)-map \(\Bl_{Z_{2},p^{*}\hspace{-0.12em}\ms{F}}(X_{2})\ra Y_{1}\), i.e.\ an \(X_{2}\)-map \(s\co\Bl_{Z_{2},p^{*}\hspace{-0.12em}\ms{F}}(X_{2})\ra Y_{2}\). By universality \(r\) and \(s\) are inverse isomorphisms.
\end{proof} 
Assume (for simplicity) that \(\ms{F}\) has a constant rank \(r\) and let \(K(X)\) denote the sheaf of meromorphic functions; cf. \cite{kle:79a}, \cite[\href{http://stacks.math.columbia.edu/tag/01X2}{Definition 01X2}]{SP} and \cite[\href{http://stacks.math.columbia.edu/tag/02OV}{Lemma 02OV}]{SP}. If \(r=1\) let \(\ms{F}^{n}\) be the image of the natural map \(\ms{F}^{\ot n}\ra i_{*}(\ms{F}^{\ot n}_{|U})\). Then 
\begin{equation}\label{eq.rk1}
\Bl_{\ms{F}}(X)\cong\Proj\bigl(\bigoplus_{n\geq 0} \ms{F}^{n}\bigr)
\end{equation}
is the scheme-theoretic closed image of \(U\) in \(\BB{P}(\ms{F})\). Oneto and Zatini observed that the Pl{\"u}cker embedding of the Grassmann gives the fractional ideal sheaf
\begin{equation}\label{eq.OZ}
\llb \ms{F}\rrb=\im\bigl\{\textstyle\bigwedge^{r}\hspace{-0.2em}\ms{F}\ra \textstyle\bigwedge^{r}\hspace{-0.3em}\ms{F}\hspace{0.15em}\ot_{\Q_{X}}K(X)\cong K(X)\bigr\}
\end{equation}
for the blowing-up \(f\co\Bl_{\ms{F}}(X)\ra X\); cf. \cite[1.4, 3.1]{one/zat:91}, \cite[3.3]{vil:06}. 
Villamayor has given an explicit description of an equivalent ideal.
Suppose \(X=\Spec A\) for a ring \(A\) and \(\ms{F}\) is given by an \(A\)-module \(M\).
Choose \(n\) generators for \(M\) and let \(\Syz(M)\) denote the kernel of the resulting map \(A^{\oplus n}\ra M\). Then \(\rk\Syz(M)=n-r\) and any choice of \(n-r\) elements in \(\Syz(M)\) which induces generators for \(K(A)\ot\Syz(M)\cong K(A)^{\oplus n-r}\) defines a linear map \(\psi\co A^{\oplus n-r}\ra A^{\oplus n}\) such that the ideal of maximal minors of \(\psi\) is isomorphic to \(\llb M\rrb\). 
See \cite[3.3]{vil:06}.
 
Curto and Morrison defines a `Grassmann blowup' as the closure in \(\BB{C}^{N}\times \Grass(n-r,n)\) of a set defined in terms of the smooth locus and the presentation matrix \(\phi\). 
In the case of a matrix factorisation of a hypersurface they state in \cite[2.1]{cur/mor:13} a universal property for the normalization of the Grassmann blowup for `birational' maps \(h\co Y\ra X\) such that \(h^{\tri}M\) is locally free. 
By our discussion and Proposition \ref{prop.modific} it follows that their normalized Grassmann blowup equals \(\Bl_{M}(X)\) for RDPs once we know that \(\Bl_{M}(X)\) is normal. Normality is not obvious and will be proved for a reflexive module on a rational surface singularity in Proposition \ref{prop.normal}.
%%%%%%%%%%%%%%%%%%%%%%%%%%%%%%%%%%%%%%%%%%%%%%%%%%%%%%
\subsection{Strict transforms and Chern classes}
The strict transform of a reflexive sheaf along a resolution of a rational surface singularity is locally free; see \cite[2.10]{gon-spr/ver:83} for quotient singularties,
the general case is cited in \cite[1.1]{art/ver:85}. Esnault proves a characterisation of sheaves on the resolution which are strict transforms of reflexive modules in \cite[2.2]{esn:85}. We give the following natural generalisation of Esnault's result which needs a slightly different proof.
%%%%%%%%%%%%%%%%%%%%%%%%%%%%%%%%%%%%%%%%%%%%%%%%%%%%%% 
%%%%%%%%%% PROP %%%%%%%%%%%%%%%%%%%%%%%%%%%%%%%%%%%%%%
\begin{prop}\label{prop.lCM}
Let \(f\co Y\ra X\) be a partial resolution of a rational surface singularity\textup{.}
\begin{enumerate}[leftmargin=2.4em, label=\textup{(\roman*)}]
\item Suppose \(M\) is a reflexive \(\Q_{X}\)-module\textup{.} Then the strict transform \(f^{\tri}M\) is a reflexive \(\Q_{Y}\)-module generated by global sections, the natural map \(M\ra f_{*}f^{\tri}M\) is an isomorphism\textup{,} and \(\mr{R}^{1}\hspace{-0.2em}f_{*}\shm{}{Y}{f^{\tri}M}{\omega_{Y}}=0\)\textup{.} In particular\textup{,} \(f^{\tri}M\) is locally free if \(Y\) is regular\textup{.}
\item If \(\ms{F}\) is a reflexive \(\Q_{Y}\)-module with \(\mr{R}^{1}\hspace{-0.2em}f_{*}\shm{}{Y}{\ms{F}}{\omega_{Y}}=0\) then \(f_{*}\ms{F}\) is a reflexive \(\Q_{X}\)-module\textup{.} Moreover\textup{,} if \(\ms{F}\) is generated by global sections then the natural map \(f^{\tri}f_{*}\ms{F}\ra \ms{F}\) is an isomorphism\textup{.} 
\end{enumerate}
\end{prop}
\begin{proof}
(i) Put \(\ms{M}=f^{\tri}M\). As a quotient of \(f^{*}M\), \(\ms{M}\) is generated by global sections. Let \(U\) denote the non-singular locus in $X$. 
Since $f$ is an isomorphism above $U$ and $f_*\ms{M}$ is torsion free, the natural map $\alpha\co M\ra f_*\ms{M}$ is an isomorphism by \cite[\href{https://stacks.math.columbia.edu/tag/0AVS}{Lemma 0AVS}]{SP}.  
Also note that \(\ms{M}\) is locally free on the complement of a \(0\)-dimensional locus since \(\ms{M}\) is torsion free and \(Y\) is normal; cf. \cite[Chap. VII, \S 4.9, Thm. 6]{bou:98}.

The duality theorem \cite[VII 3.4]{har:66} (cf. \cite[3.4.4]{con:00}) gives an isomorphism:
\begin{equation}\label{eq.dual}
\BB{R}f_{*}\BB{R}\shm{}{Y}{\ms{M}}{\omega_{Y}}\xra{\hspace{1.5em}\sim\hspace{1.5em}}\BB{R}\shm{}{X}{\BB{R}f_{*}\ms{M}}{\omega_{X}}
\end{equation}
Rationality gives \(\BB{R}f_{*}\ms{M}\simeq f_{*}\ms{M}\) and the resulting spectral sequence gives short exact sequences:
\begin{equation}\label{eq.Mdual}
0\ra \mr{R}^{1}\hspace{-0.2em}f_{*}\sxt{p-1}{Y}{\ms{M}}{\omega_{Y}}\lra \sxt{p}{X}{M}{\omega_{X}}\lra f_{*}\sxt{p}{Y}{\ms{M}}{\omega_{Y}}\ra 0
\end{equation}
Since \(M\) is maximal Cohen-Macaulay, \(\sxt{p}{X}{M}{\omega_{X}}=0\) for all \(p>0\) which implies \(\sxt{p}{Y}{\ms{M}}{\omega_{Y}}=0\) because  
\(\sxt{p}{Y}{\ms{M}}{\omega_{Y}}\) has zero dimensional support for \(p>0\). It follows that \(\ms{M}\) is maximal Cohen-Macaulay, i.e.\ reflexive since \(Y\) is normal. Moreover, \(\mr{R}^{1}\hspace{-0.2em}f_{*}\shm{}{Y}{\ms{M}}{\omega_{Y}}=0\) by \eqref{eq.Mdual}. For local cohomology; cf. \cite[Chap. 3]{bru/her:98}.

(ii) Since \(Y\) is normal, \(\ms{F}\) is maximal 
Cohen-Macaulay, so \eqref{eq.dual} gives (with \(\ms{F}\) replacing \(\ms{M}\)) an isomorphism \(\BB{R}f_{*}\shm{}{Y}{\ms{F}}{\omega_{Y}}\simeq\BB{R}\shm{}{X}{\BB{R}f_{*}\ms{F}}{\omega_{X}}\).
The associated second quadrant cohomological spectral sequence gives an exact sequence:
\begin{equation}\label{eq.Fdual}
\begin{aligned}
0\ra{} &\sxt{1}{X}{\mr{R}^{1}\hspace{-0.2em}f_{*}\ms{F}}{\omega_{X}}\ra f_{*}\shm{}{Y}{\ms{F}}{\omega_{Y}}\ra\shm{}{X}{f_{*}\ms{F}}{\omega_{X}}\\
\ra{} & \sxt{2}{X}{\mr{R}^{1}\hspace{-0.2em}f_{*}\ms{F}}{\omega_{X}}\ra \mr{R}^{1}\hspace{-0.2em}f_{*}\shm{}{Y}{\ms{F}}{\omega_{Y}}\ra\sxt{1}{X}{f_{*}\ms{F}}{\omega_{X}}\ra \dots
\end{aligned}
\end{equation}
Since \(\mr{R}^{q}\hspace{-0.2em}f_{*}\shm{}{Y}{\ms{F}}{\omega_{Y}}=0\) for \(q>0\), \eqref{eq.Fdual} gives 
\begin{equation}
\sxt{q}{X}{f_{*}\ms{F}}{\omega_{X}}\cong \sxt{q+2}{X}{\mr{R}^{1}\hspace{-0.2em}f_{*}\ms{F}}{\omega_{X}}\qquad (q>0) 
\end{equation}
and the latter is zero by \cite[3.5.11]{bru/her:98}, 
i.e.\ \(f_{*}\ms{F}\) is maximal Cohen-Macaulay. Any map \(\Q_{Y}^{\oplus n}\ra \ms{F}\) factors as 
\begin{equation}\label{eq.glob}
\Q_{Y}^{\oplus n}\cong f^{\tri}f_{*}\Q_{Y}^{\oplus n}\lra f^{\tri}f_{*}\ms{F}\xra{\hspace{0.3em}\rho\hspace{0.3em}} \ms{F}
\end{equation}
hence if the former is surjective so is \(\rho\). But since \(f^{\tri}f_{*}\ms{F}\) is torsion free, \(\rho\) is an isomorphism.
\end{proof}
\begin{rem}
The argument in (ii) works for any normal surface singularity. See also \cite[2.74]{kol:13}.
\end{rem}
%%%%%%%%%%%%%% LEMMA %%%%%%%%%%%%%%%%%%%%%%%%%%%%%%%%%
\begin{lem}\label{lem.chern}
Suppose \(f\co Y\ra X\) is a partial resolution of a rational surface singularity and \(\ms{F}\) is a locally free \(\Q_{Y}\)-module of rank \(r\) generated by global sections\textup{.} 
A generic choice of \(r\) global sections gives a short exact sequence of coherent \(\Q_{Y}\)-modules
\begin{equation*}
\alpha\co\hspace{0.6em}0\ra \Q_{Y}^{\oplus r}\xra{(s_{1},\dots,s_{r})} \ms{F}\lra\Q_{D}\ra 0
\end{equation*}
where \(D\) is an effective\textup{,} affine\textup{,} smooth divisor intersecting \(E(f)_{\tn{red}}\) transversally\textup{.} 

Moreover\textup{,} the \(r-1\) sections \(s_{2},\dots,s_{r}\) give a short exact sequence
\begin{equation*} 
\beta\co\hspace{0.6em}0\ra \Q_{Y}^{\oplus r-1}\xra{(s_{2},\dots,s_{r})} \ms{F}\xra{\,\,w\,\,}\textstyle\bigwedge^{\hspace{-0.12em}r}\hspace{-0.2em}\ms{F}\ra 0
\end{equation*}
where \(w(m)=m\wedge s_{2}\wedge\dots\wedge  s_{r}\) and
\(\bigwedge^{\hspace{-0.12em}r}\hspace{-0.2em}\ms{F}\cong \Q_{Y}(D)\)\textup{.}
\end{lem}
%%%%%%%%%%%%%
\begin{proof} 
By Proposition \ref{prop.pic} the prime components of \(E(f)_{\tn{red}}\) are smooth. Then \(\alpha\) follows as in \cite[1.2]{art/ver:85}.
Pushout of $\Q_{Y}^{\oplus r}\ra \ms{F}$ along the first  projection $\Q_{Y}^{\oplus r}\ra \Q_Y$ gives a s.e.s. $0\ra \Q_{Y}^{\oplus r-1}\hspace{-0.2em}\ra \ms{F}\xra{\,p\,}\ms{E}\ra 0$ where $\ms{E}$ is an invertible sheaf by \cite[3.5.1]{vdber:04}. Since $\im(s_{2},\dots,s_{r})\sbeq\ker w$ there is an induced map $i\co \ms{E}\ra \bigwedge^{\hspace{-0.12em}r}\hspace{-0.2em}\ms{F}$ with $w=ip$. 
The map \(w\) is surjective since \(p\) splits locally. Then $i$ is an isomorphism. Applying $\shm{}{Y}{-}{\Q_Y}$ to the induced s.e.s. $0\ra \Q_{Y}\ra \ms{E}\ra\Q_D\ra 0$ gives the s.e.s. $0\ra \ms{E}^{\vee}\ra\Q_Y\ra\Q_D\ra 0$ which implies that $\ms{E}\cong \Q_Y(D)$.
\end{proof}
For a locally free sheaf \(\ms{F}\) of rank \(r\) we use the notation \(\ch_{1}(\ms{F})=\bigwedge^{\hspace{-0.12em}r}\hspace{-0.2em}\ms{F}\). Note that Wunram in \cite[A2]{wun:88} gave two non-isomorphic indecomposable reflexive modules of rank $3$ on $\BB{I}_7$ with equal Chern classes.
%%%%%%%%%%%%%%%%%%%%%%%%%%%%%%%%%%%%%%%%%%%%%%%%%%%%%%
\subsection{Base change and cohomology}
We will need a base change result for $\Ext$ which is not covered by \cite[7.7.5]{EGAIII2}.
Let \(f\co Y\ra X=\Spec R\) and \(g\co X\ra S=\Spec A\) be maps of schemes and \(\ms{E}\) and \(\ms{F}\) coherent \(\Q_{Y}\)-modules such that \(Y\), \(\ms{E}\) and \(\ms{F}\) are \(S\)-flat. Assume that \(g\) is local (i.e.\ given by a local map of local \(k\)-algebras \(A\ra R\)) and $f$ is proper. Put \(\pi=gf\).  
For any quasi-coherent $\Q_Y$-module $\ms{G}$ and any $n$, $\sxt{n}{Y}{\ms{E}}{\ms{G}}$ is a quasi-coherent $\Q_Y$-module, moreover, $\pi_*\sxt{n}{Y}{\ms{E}}{\ms{G}}$ is quasi-coherent since $\pi$ is proper; \cite[\href{https://stacks.math.columbia.edu/tag/01XJ}{Lemma 01XJ}]{SP}. Also note that $\xt{n}{Y}{\ms{E}}{\ms{G}}$ is naturally an $R$-module which is finitely generated if $\ms{G}$ is coherent by the local-to-global spectral sequence $\cE^{p,q}_2=\cH^q(\sxt{p}{Y}{\ms{E}}{\ms{G}})\Ra \xt{n}{Y}{\ms{E}}{\ms{G}}$ and properness (\cite[3.2.1]{EGAIII1}). The natural isomorphism of functors $f_{\hspace{-0.08em}*}\shm{}{Y}{\ms{E}}{-}\cong\hm{}{Y}{\ms{E}}{-}^{\tilde{}}$ extends to an isomorphism of the right derived universal $\delta$-functors: 
\begin{equation}
\left\{\sxt{n}{f}{\ms{E}}{-}\cong\xt{n}{Y}{\ms{E}}{-}^{\tilde{}}\,\right\}_{n\in \BB{Z}}\co \cat{QCoh}(Y)\lra \cat{QCoh}(X)
\end{equation}
which restricts to functors of coherent sheaves $\cat{Coh}(Y)\ra \cat{Coh}(X)$.

For every integer $n$ we define a functor of quasi-coherent sheaves 
\begin{equation}
F^n\co \cat{QCoh}(S)\lra\cat{QCoh}(X)\quad \tn{by}\quad F^{n}(I)=\xt{n}{Y}{\ms{E}}{{\ms{F}}\ot \pi^{*}\hspace{-0.1em}I}^{\tilde{}}\,.
\end{equation}
The functor given by $I\mapsto \ms{F}\ot \pi^{*}I$ is exact since $\ms{F}$ is $S$-flat and $\{F^n\}_{n\in \BB{Z}}$ is a cohomological $\delta$-functor. Moreover, $F^n(I)$ is a coherent $\Q_X$-module if $I$ is a coherent $\Q_S$-module, and $F^n$ commutes with filtered direct limits.  
Hence the conditions in \cite[5.1-2]{ogu/ber:72} are satisfied and the conclusions apply to the exchange maps 
\begin{equation}
e^n_I\co F^n(\Q_S)\ot_{\Q_{\hspace{-0.08em}X}} g^{*}\hspace{-0.1em}I\lra F^n(I)
\end{equation}
which are defined essentially by applying $F^n$ to the multiplication maps $\cdot u\co \Q_S\ra I$ for $u\in I$, see the beginning of Section 4 in \cite{ogu/ber:72} or \cite[7.2.2]{EGAIII2}. 

We first extend the exchange map to ordinary fibre products by a local scheme map \(p\co T=\Spec B\ra S\). Put $X':=X{\times}_ST$ and $Y':=Y{\times}_ST$. Let $\mr{pr}_{\hspace{-0.12em}X}\co X'\ra X$, $q\co Y'\ra Y$, $g'\co X'\ra T$, $f'\co Y'\ra X'$ and $\pi'=g'{\ccirc} f'$ denote the projections. Suppose $\ms{G}$ is a quasi-coherent $\Q_{Y'}$-module. Applying $\BB{R}f_*$ to the natural, functorial isomorphism in \cite[II 5.10]{har:66} gives
\begin{equation}\label{eq.dad}
\BB{R}(\mr{pr}_X)_*\BB{R}\shm{}{f'}{\BB{L}q^{*}\ms{E}}{\ms{G}}
\simeq
\BB{R}\shm{}{f}{\ms{E}}{\BB{R}q_{*}\ms{G}}\,.
\end{equation}
Note that $\BB{L}q^{*}\ms{E}\simeq q^{*}\ms{E}$ since $\ms{E}$ is $S$-flat. Moreover, $q$ and $\mr{pr}_X$ are affine, so \eqref{eq.dad} gives isomorphisms 
\begin{equation}
\eta\co \xt{n}{Y}{\ms{E}}{q_{*}\ms{G}}^{\tilde{}}\cong (\mr{pr}_{\hspace{-0.12em}X}\hspace{-0.18em})_{*}\xt{n}{Y'}{q^{*}\ms{E}}{\ms{G}}^{\tilde{}}\,.
\end{equation}
Suppose now that \(I\) is a quasi-coherent \(\Q_{T}\)-module and let $e_I^n$ denote the exchange map $e_I^n\co \xt{n}{Y'}{q^*\ms{E}}{q^*\ms{F}}^{\tilde{}}\,\ot_{\Q_{X'}}(g')^*I\ra \xt{n}{Y'}{q^*\ms{E}}{q^*\ms{F}\ot_{\Q_{Y'}}(\pi')^*\hspace{-0.12em}I}^{\tilde{}}$. 
We define the (ordinary) base change map $b^n_I$ by the following commutative diagram
\begin{equation}
\xymatrix@C+15pt@R-8pt@H-6pt{
\mr{pr}_{\hspace{-0.12em}X}^*\xt{n}{Y}{\ms{E}}{\ms{F}}^{\tilde{}}\,\ot_{\Q_{X'}}(g')^*\hspace{-0.12em}I\ar[r]^(0.48){b^n_I} \ar[d]_(0.45){a}
& \xt{n}{Y'}{q^*\ms{E}}{q^*\ms{F}\ot_{\Q_{Y'}}(\pi')^*\hspace{-0.12em}I}^{\tilde{}}
\\
\mr{pr}_{\hspace{-0.12em}X}^*\xt{n}{Y}{\ms{E}}{q_*q^*\ms{F}}^{\tilde{}}\,\ot_{\Q_{X'}}(g')^*\hspace{-0.12em}I\ar[r]^(0.51){\eta^{\mr{ad}}\ot\id} 
& \xt{n}{Y'}{q^*\ms{E}}{q^*\ms{F}}^{\tilde{}}\,\ot_{\Q_{X'}}(g')^*\hspace{-0.12em}I \ar[u]_{e^n_I}
}
\end{equation}
where $a$ is induced by the canonical map $\ms{F}\ra q_*q^*\ms{F}$.

To fit our application we assume \(R\) is henselian. Let \(g_{T}\co X_{T}=\Spec (R\ot_{\hspace{-0.08em}A}B)^{\tn{h}}\ra T\) denote the projection where the \(\tn{h}\) denotes henselisation in the canonical \(k\)-point. Let \(f_{T}\co Y_{T}\ra X_{T}\) denote the (ordinary) pullback of \(f\) to \(X_{T}\) and let $p_X\co X_T\ra X$ and $p_Y\co Y_T\ra Y$ denote the induced projections. Put \(\pi_{T}=g_{T}f_{T}\), \(\ms{F}_{T}=p_Y^{*}\ms{F}\), and so on. Let $h\co X_T\ra X'$ denote the henselisation map and $h_Y\co Y_T\ra Y'$ the  pullback of $h$. Flat base change by $h$ gives a canonical isomorphism (e.g. by Lazard's theorem \cite[\href{https://stacks.math.columbia.edu/tag/058G}{Theorem 058G}]{SP}, \cite[\href{https://stacks.math.columbia.edu/tag/07TB}{Lemma 07TB}]{SP} and the local to global spectral sequence):
\begin{equation}\label{eq.hbc}
\begin{aligned} 
h^*\xt{n}{Y'}{q^*\hspace{-0.2em}\ms{E}}{q^*\hspace{-0.2em}\ms{F}\ot_{\Q_{Y'}}(\pi')^*I}^{\tilde{}}
& \cong\xt{n}{Y_T}{\ms{E}_T}{\ms{F}_T\ot_{\Q_{Y_T}} \pi_T^*I}^{\tilde{}}
\end{aligned}
\end{equation}
There is also an isomorphism of $\Q_{X_T}$-modules 
\begin{equation}
s\co\xt{n}{Y}{\ms{E}}{\ms{F}}_{T}^{\tilde{}}\hspace{0.06em}\ot_{\Q_{X_T}} g_T^*I\xra{\;\;\simeq\;\;}
h^*\hspace{-0.18em}\left[\mr{pr}_X^*\xt{n}{Y}{\ms{E}}{\ms{F}}^{\tilde{}}\,\ot_{\Q_{X'}}(g')^*I\right].
\end{equation}
Define the \(\Q_{X_{T}}\)-linear (henselian)
\emph{base change map} 
\begin{equation}\label{eq.c}
c^{n}_{I}\co \xt{n}{Y}{\ms{E}}{\ms{F}}_{T}^{\tilde{}}\hspace{0.06em}\ot_{\Q_{X_T}} g_T^*I 
\lra 
\xt{n}{Y_{T}}{\ms{E}_{T}}{{\ms{F}_{T}}\hspace{0.12em}\ot_{\Q_{Y_T}}\hspace{0.12em}\pi_{T}^{*}I}^{\tilde{}}
\end{equation}
as the composition of $h^*(b^n_{I})\ccirc s$ with \eqref{eq.hbc}.
Put \(X_{0}=X{\times}_{S}\Spec k\), $Y_0=Y{\times}_XX_0$, let \(\ms{E}_{0}\) denote the pullback of \(\ms{E}\) to \(Y_{0}\), and so on.
%%%%%%%%%%%%%%%%%%% PROPOSITION %%%%%%%%%%%%%%%%%%%%%%%%%%%%
\begin{prop}\label{prop.extbc}
Assume the base change map
\begin{equation*}
c^{n}_{k}\co\xt{n}{Y}{\ms{E}}{\ms{F}}^{\tilde{}}_0\lra \xt{n}{Y_{0}}{\ms{E}_{0}}{\ms{F}_{0}}^{\tilde{}}
\end{equation*}
is surjective\textup{.} Then\textup{:}
\begin{enumerate}[leftmargin=2.4em, label=\textup{(\roman*)}]
\item For all local maps \(T\ra S\) and quasi-coherent \(\Q_{T}\)-modules \(I\)\textup{,} the base change map \(c_{I}^{n}\) is an isomorphism\textup{.}
\item The following statements are equivalent\textup{:}
\begin{enumerate}[leftmargin=1.8em, label=\textup{(\alph*)}]
\item \(c_{k}^{n-1}\) is surjective\textup{.}
\item The \(\Q_{X}\)-module \(\xt{n}{Y}{\ms{E}}{\ms{F}}^{\tilde{}}\) is \(S\)-flat\textup{.}
\end{enumerate}
\end{enumerate}
\end{prop}
\begin{proof}
We first establish a compatibility of $e^n_{p_*I}$ with $(\mr{pr}_{\hspace{-0.12em}X})_*(e^n_I)$. There is a natural isomorphism $\tau\co \ms{F}\ot_{\Q_Y}\pi^*p_*I\cong q_*(q^*\ms{F}\ot_{\Q_{Y'}}(\pi')^*\hspace{-0.12em}I)$ with adjoint $\tau^{\tn{ad}}$. Note that the canonical map $\ms{F}\ot\pi^*p_*I\ra q_*q^*(\ms{F}\ot\pi^*p_*I)$ composed with 
\begin{equation}
q_*\tau^{\tn{ad}}\co q_*q^*(\ms{F}\ot\pi^*p_*I)\lra q_*(q^*\ms{F}\ot(\pi')^*I)
\end{equation}
equals $\tau$. Let $u$ be an element in $I$ and let $\cdot u$ denote the map $\Q_S\ra p_*I$. To simplify the notation we also write $\cdot u$ for some of the induced maps like ${\id}\ot\pi^*(\cdot u)\co \ms{F}\ra \ms{F}\ot \pi^*p_*I$.
There is a diagram of $\Q_X$-linear maps:
\begin{equation}\label{eq.comp}
\xymatrix@C-9pt@R-8pt@H-6pt{
\xt{n}{Y}{\ms{E}}{\ms{F}}\ar[d]^(0.45){(\cdot u)_*}\ar[r]^(0.43){a} &
\xt{n}{Y}{\ms{E}}{q_*q^*\ms{F}}\ar[d]^(0.45){(q_*q^*(\cdot u))_*}\ar[r]^(0.48){\eta} &
\xt{n}{Y'}{q^*\ms{E}}{q^*\ms{F}}\ar[d]^(0.45){(q^*(\cdot u))_*}
\\
\xt{n}{Y}{\ms{E}}{\ms{F}\ot \pi^*p_*I}\ar[dr]_(0.45){\tau_*}\ar[r]^(0.43){a} &
\xt{n}{Y}{\ms{E}}{q_*q^*(\ms{F}\ot\pi^*p_*I)}\ar[d]^(0.44){(q_*\tau^\tn{ad})_*}\ar[r]^(0.48){\eta} &
\xt{n}{Y'}{q^*\ms{E}}{q^*(\ms{F}\ot\pi^*p_*I)}\ar[d]^(0.44){(\tau^\tn{ad})_*}
\\
&
\xt{n}{Y}{\ms{E}}{q_*(q^*\ms{F}\ot(\pi')^*I)}\ar[r]^(0.49){\eta} &
\xt{n}{Y'}{q^*\ms{E}}{q^*\ms{F}\ot(\pi')^*I}
}
\end{equation}
Since $\eta$ is functorial the diagram commutes. The composition $\tau^\tn{ad}\ccirc q^*(\cdot u)$ is the multiplication map $\cdot u\co q^*\ms{F}\ra q^*\ms{F}\ot (\pi')^*I$. 
There is a natural isomorphism 
\begin{equation}
\gamma\co (\mr{pr}_{\hspace{-0.12em}X})_*\hspace{-0.2em}\left[\mr{pr}_{\hspace{-0.12em}X}^*\xt{n}{Y}{\ms{E}}{\ms{F}}\ot_{\Q_{X'}}(g')^*\hspace{-0.12em}I\right]\cong \xt{n}{Y}{\ms{E}}{\ms{F}}\ot_{\Q_{X}}g^*\hspace{-0.12em}p_*I.
\end{equation}
With the compatibility in \eqref{eq.comp} one shows that $(\mr{pr}_{\hspace{-0.12em}X}\hspace{-0.12em})_*b^n_I
=(\mr{pr}_{\hspace{-0.12em}X}\hspace{-0.12em})_*[e^n_I{\ccirc}(\eta^{\tn{ad}}\ot \id){\ccirc} a]
$ equals $\eta\ccirc\tau_*\ccirc e^n_{p_*I}\ccirc\gamma$ where $\gamma$, $\tau_*\co \xt{n}{Y}{\ms{E}}{\ms{F}\ot\pi^*p_*I}\ra \xt{n}{Y}{\ms{E}}{q_*(q^*\ms{F}\ot(\pi')^*I}$ and $\eta$ are isomorphisms.  
In the condition $I$ is $\Q_{\Spec k}$ (denoted by $k$), $p$ is the closed embedding $T=\Spec k\ra S$, $g^*p_*k$ equals $\Q_{X_0}$ and $\xt{n}{Y}{\ms{E}}{\ms{F}}_{T}\hspace{0.06em}\ot_{\Q_{X_T}} g_T^*k$ is isomorphic to  $\xt{n}{Y}{\ms{E}}{\ms{F}}_0$. Since $X$ is henselian so is $X_0$ and $c_k^n=b_k^n$. Hence the assumption is equivalent to $e^n_{p_*k}$ being surjective. Finally, $p_*I$ is a quasi-coherent $\Q_S$-module \cite[\href{https://stacks.math.columbia.edu/tag/01XJ}{Lemma 01XJ}]{SP}. By \cite[5.1.2']{ogu/ber:72}, $e^n_{p\hspace{-0.05em}_*\hspace{-0.1em}I}$ is an isomorphism and then so is $b_I^n$ and $c^n_I$.
 
For (ii), $c_k^{n-1}$ is surjective if and only if $e_{p_*k}^{n-1}$ is surjective if and only if $F^n(\Q_S)$ is $S$-flat by \cite[5.2]{ogu/ber:72}. But $F^n(\Q_S)=\xt{n}{Y}{\ms{E}}{\ms{F}}$.
\end{proof}
The expression \emph{$\xt{n}{Y}{\ms{E}}{\ms{F}}$ commutes with base change} (or similar) means that the conclusion in Proposition \ref{prop.extbc} (i) holds.
\begin{ex}
Since \(\xt{n}{Y}{\Q_{Y}}{\ms{F}}\cong \mr{H}^{n}(Y,\ms{F})\), Proposition \ref{prop.extbc} gives a variant of global cohomology and base change without simultaneous properness and flatness (so apparently not covered by \cite[7.7.5]{EGAIII2}). For artinian base, see Wahl's \cite[0.4]{wah:76}.
\end{ex}
%%%%%%%%%%%%%%%%%%% COROLLARY %%%%%%%%%%%%%%%%%%%%%%%%%%%%
\begin{cor}\label{cor.extbc}
Assume $\xt{n+1}{Y_{0}}{\ms{E}_{0}}{\ms{F}_{0}}=0$\textup{.} Then $\xt{n}{Y}{\ms{E}}{\ms{F}}$ commutes with base change\textup{.} If furthermore $c^{n-1}_k$ is surjective, then $\xt{n}{Y}{\ms{E}}{\ms{F}}^{\tilde{}}$ is $S$-flat and hence a deformation of $\xt{n}{Y_{0}}{\ms{E}_{0}}{\ms{F}_{0}}^{\tilde{}}$\,\textup{.}
\end{cor}
\begin{proof}
Since $\xt{n+1}{Y_{0}}{\ms{E}_{0}}{\ms{F}_{0}}=0$, $c^{n+1}_k$ is surjective and by Proposition \ref{prop.extbc} (i) an isomorphism. Since $\xt{n+1}{Y}{\ms{E}}{\ms{F}}^{\tilde{}}$ is coherent, Nakayama's lemma implies that $\xt{n+1}{Y}{\ms{E}}{\ms{F}}^{\tilde{}}=0$ and in particular is $S$-flat. By Proposition \ref{prop.extbc} (ii), $c^n_k$ is surjective and by Proposition \ref{prop.extbc} (i), $\xt{n}{Y}{\ms{E}}{\ms{F}}
$ commutes with base change. If in addition $c^{n-1}_k$ is surjective, then $\xt{n}{Y}{\ms{E}}{\ms{F}}^{\tilde{}}$ is $S$-flat by Proposition \ref{prop.extbc} (ii).
\end{proof}
%%%%%%%%%%%%%%%%%%%%%%%%%%%%%%%%%%%%%%%%%%%%%%%%%%%%%%
%%%%%%%%%%%%%% SECTION %%%%%%%%%%%%%%%%%%%%%%%%%%%%%%%
%%%%%%%%%%%%%%%%%%%%%%%%%%%%%%%%%%%%%%%%%%%%%%%%%%%%%%
\section{Deformations of pairs with partial resolutions}\label{sec.def}
A pair \((X,\ms{F})\) is a scheme \(X\) and a coherent \(\Q_{X}\)-module \(\ms{F}\). A map of pairs \((X_{2},\ms{F}_{2})\ra (X_{1},\ms{F}_{1})\) is a scheme map \(p\co X_{2}\ra X_{1}\) and a map of \(\Q_{X_{2}}\)-modules \(\alpha\co p^{*}\ms{F}_{1}\ra \ms{F}_{2}\). 
One obtains a category of pairs. 
If \(X\ra S\) is a scheme map then the pair \((X,\ms{F})\) is flat over \(S\) if \(X\) and \(\ms{F}\) both are \(S\)-flat.
Let \(\cat{H}_{k}\) be the category of affine schemes \(S=\Spec A\) above \(\Spec k\) where \(A\) is a (noetherian) local henselian \(k\)-algebra. 
Fix a singularity \(X_{0}=\Spec B_{0}\) and a coherent \(\Q_{X_{0}}\)-module \(M_{0}\). There is a fibred category \(\cdf{}{(X_{0},M_{0})}/\cat{H}_{k}\) of deformations of the pair; extensions
\begin{equation}\label{eq.def}
(X_{0},M_{0})\lra (X,M)
\end{equation}
flat over \(\Spec k\ra S\) in \(\cat{H}_{k}\) where \(X=\Spec B\) is assumed to be algebraic over \(S\), i.e.\ \(B\) is given as the henselisation of a finite type \(A\)-algebra in a closed point.
A morphism in \(\cdf{}{(X_{0},M_{0})}\) above a map \(S'\ra S\) in \(\cat{H}_{k}\) is a map of pairs \((p,\alpha)\co(X',M')\ra(X,M)\) above \((X_{0},M_{0})\), such that the map of schemes is cartesian in the category of henselian local schemes:
\begin{equation}\label{eq.hsquare}
\xymatrix@C+6pt@R-8pt@H-6pt{
X'\ar[d]\ar[r]^{p} & X\ar[d] \\
S'\ar@{}[ur]|(0.53){\hspace{0.12em}\Box^{\tn{h}}}\ar[r] & S
}
\end{equation}
and \(\alpha\co p^{*}M\ra M'\) is an isomorphism. Given a deformation \eqref{eq.def} and a map \(S'\ra S\) in \(\cat{H}_{k}\) there exists a base change as in \eqref{eq.hsquare} and so the cartesian property holds. Identifying isomorphic objects defines a deformation functor \(\df{}{(X_{0},M_{0})}\co \cat{H}_{k}\ra \Sets\).
 
We need conditions on \((X_{0},M_{0})\) and a birational map \(Y_{0}\ra X_{0}\) that imply the conditions in Corollary \ref{cor.blowup} for all deformations. 
%%%%%%%%%%%%%%%%%%% LEMMA %%%%%%%%%%%%%%%%%%%%%%%%%%%%
\begin{lem}\label{lem.U}
Assume \(X_{0}\) is integral and \(M_{0}\) is torsion free. There is an \(M_{0}\)-regular element \(0\neq t_{0}\in\Gamma(\Q_{X_{0}})\) with \(U_{0}:=D(t_{0})\) such that \(M_{0}{}_{|U_{0}}\) is locally free\textup{.} Let \((X, M)\) be a deformation of \((X_{0}, M_{0})\)\textup{.} 
\begin{enumerate}[leftmargin=2.4em, label=\textup{(\roman*)}]
\item For any lifting \(t\in\Gamma(\Q_{X})\) of \(t_{0}\)\textup{,} the open subscheme \(U=D(t)\) is dense in \(X\) and \(M_{|U}\) is locally free\textup{.} In particular\textup{,} for any \(p\) as in \eqref{eq.hsquare}, \(p^{-1}(U)\) is dense in \(X'\)\textup{.}
\end{enumerate}
Moreover\textup{,} let \(f\co Y\ra X\) be a proper scheme map with \(Y\) \(S\)-flat such that the central fibre \(f_{0}\co Y_{0}\ra X_{0}\) is an isomorphism above \(U_{0}\) and \(t_{0}\) defines a Cartier divisor on \(Y_{0}\)\textup{.} 
\begin{enumerate}[resume*]
\item The map \(f\) is an isomorphism above \(U\) and \(f^{-1}(U)\) is dense in \(Y\)\textup{.}
\end{enumerate}  
\end{lem}
\begin{proof}
By \cite[Chap. II, \S 5.1, Prop. 2]{bou:98} there exists a \(t_{0}\) such that \(M_{0}{}_{|U_{0}}\) is locally free. Put \(B=\Gamma(\Q_{X})\). By \cite[19.2.4]{EGAIV4}, \(t\) is \(B\)-regular and \(M\)-regular and in particular \(U\) is dense in \(X\). A choice of \(n\) generators gives a surjection \(\alpha\co B_{t}^{\oplus n}\ra M_{t}\). Since $M_t\ot k\cong (M_0)_{t_0}$ is free, $\xt{1}{U_0}{M_t\ot k}{\ker\alpha\ot k}=0$. Then \(\xt{1}{B_{t}}{M_{t}}{\ker\alpha}=0\) by Proposition \ref{prop.extbc} (i) and Nakayama's lemma. Hence $\alpha$ splits. Since the image \(t'\in\Gamma(\Q_{X'})\) of \(t\) lifts \(t_{0}\), \(p^{-1}(U)=D(t')\) is dense as above.

For (ii) note that \(t\) as global section of \(Y\) defines a Cartier divisor with complement \(f^{-1}(U)\) which hence is dense in \(Y\). Consider a closed point \(x\) in \(U\), i.e. \(x\in U_{0}\), with \(y\in f_{0}^{-1}(U_{0})\) the unique preimage of \(x\). Then \(f\) is flat at \(y\) by \cite[11.3.10]{EGAIV3} and {\'e}tale by \cite[17.6.3 e]{EGAIV4}. {\'E}tale and proper implies finite and Nakayama's lemma implies \(f\) is an isomorphism above \(U\).   
\end{proof} 
%%%%%%%%%%%%%%%%%%% DEFINITION %%%%%%%%%%%%%%%%%%%%%%%%%%%%
\begin{defn}\label{def.fib}
Fix a pair \((X_{0},M_{0})\) with \(X_{0}\) an integral and Cohen-Macaulay singularity and \(M_{0}\) torsion free. Let \((f_{0},\alpha_{0})\co (Y_{0},\ms{M}_{0})\ra (X_{0},M_{0})\) be a map of pairs such that \(f_{0}\) is a partial resolution with \(\dim E(f_{0})\leq 1\), \(\ms{M}_{0}\) is a coherent and locally free
\(\Q_{Y_{0}}\)-module, and the adjoint of \(\alpha_{0}\) is an isomorphism \(\alpha_{0}^{\tn{ad}}\co M_{0}\cong (f_{0})_{*}\ms{M}_{0}\). 
We also denote \((f_{0},\alpha_{0})\) by \((Y_{0}/X_{0},\ms{M}_{0}/M_{0})\).

Let \(\cdf{}{(Y_{0}/X_{0},\ms{M}_{0}/M_{0})}\) be the category where the objects are extensions of maps of pairs 
\begin{equation}\tag{{\(\ast\)}}
\xymatrix@C+6pt@R-3pt@H-6pt{
(Y_{0},\ms{M}_{0})\ar[r]\ar[d]_(0.42){(f_{0},\alpha_{0})} & (Y,\ms{M})\ar[d]^(0.42){(f,\alpha)} & 
\\
(X_{0},M_{0})\ar[r] & (X,M)
}
\end{equation}
over some \(\Spec k\ra S\) in \(\cat{H}_{k}\) with the following properties\textup{:}
\begin{enumerate}[leftmargin=2.4em, label=\textup{(\roman*)}]
\item \((Y,\ms{M})\) is flat over \(S\) and \((Y_{0},\ms{M}_{0})\cong(Y,\ms{M})\times_{X}X_{0}\)
\item \((X_{0},M_{0})\ra (X,M)\) is an object over \(\Spec k\ra S\) in \(\cdf{}{(X_{0},M_{0})}\)
\item \(f\) is proper\textup{,} \(\mr{R}^{1}\hspace{-0.15em}f_{*}\Q_{Y}\) and \(\mr{R}^{1}\hspace{-0.15em}f_{*}\ms{M}\) are \(S\)-flat
\end{enumerate}
We call \((f,\alpha)\) a deformation of \((f_{0},\alpha_{0})\).

A map \((f',\alpha')\ra(f,\alpha)\) in \(\cdf{}{(Y_{0}/X_{0},\ms{M}_{0}/M_{0})}\) over a map \(g\co S'\ra S\) in \(\cat{H}_{k}\) is a commutative diagram of deformations of \((f_{0},\alpha_{0})\)
\begin{equation}\tag{{\(\ast\ast\)}}
\xymatrix@C+6pt@R-3pt@H-6pt{
(Y',\ms{M}')\ar[r]^{(q,\beta)}\ar[d]_(0.42){(f',\alpha')} & (Y,\ms{M})\ar[d]^(0.42){(f,\alpha)} & 
\\
(X',M')\ar[r]^{(p,\gamma)} & (X,M)
}
\end{equation}
\end{defn}
%%%%%%%%%%%%%%%%%%% PROPOSITION %%%%%%%%%%%%%%%%%%%%%%%%%%%%
\begin{prop}\label{prop.fib}
The forgetful functor \(\cdf{}{(Y_{0}/X_{0},\ms{M}_{0}/M_{0})}\ra \cat{H}_{k}\) is a fibred category\textup{.}
\end{prop}
%%%%%%%%%%
\begin{proof}
Suppose \((Y/X,\ms{M}/M)\) is an object in the fibre category \(\cdf{}{(Y_{0}/X_{0},\ms{M}_{0}/M_{0})}(S)\) as in diagram (\(\ast\)) and \(g\co S'\ra S\) in \(\cat{H}_{k}\). Then a diagram (\(\ast\ast\)) has to be produced which is cartesian over \(g\) in our category.  Put \(X'=X{\times}_{S}^{\tn{h}}S'\), \(f'\co Y'\ra X'\) equal to pullback of \(f\) along the projection \(p\), and sheaves \(\ms{M}'=q^{*}\ms{M}\), \(M'=p^{*}M\). The map \(\alpha'\) is given by the composition \((f')^{*}(p^{*}M)\cong q^{*}(f^{*}M)\xra{q^{*}\alpha}q^{*}\ms{M}\). 
This gives a map of deformations (\(\ast\ast\)) with (i), (ii) and \(f'\) proper. Since \(\mr{R}^{2}(f_{0})_{*}(-)=0\), \(\mr{R}^{1}\hspace{-0.15em}f_{*}\ms{M}\) commutes with base change by Corollary \ref{cor.extbc}. It follows that \(\mr{R}^{1}(f')_{*}\ms{M}'\cong p^{*}\mr{R}^{1}\hspace{-0.15em}f_{*}\ms{M}\) is \(S'\)-flat. Similarly for \(\mr{R}^{1}(f')_{*}\Q_{Y'}\), so (iii) holds.
\end{proof}
%%%%%%%%%%%%%%%%%%% LEMMA %%%%%%%%%%%%%%%%%%%%%%%%%%%%
\begin{lem}\label{lem.fib}
Let \((f,\alpha)\co (Y,\ms{M})\ra (X,M)\) be an object in \(\cdf{}{(Y_{0}/X_{0},\ms{M}_{0}/M_{0})}(S)\)\textup{.}
\begin{enumerate}[leftmargin=2.4em, label=\textup{(\roman*)}]
\item The pair \((\Spec f_{*}\Q_{Y},f_{*}\ms{M})\) is an object in \(\cdf{}{(X_{0},M_{0})}(S)\) isomorphic to \((X,M)\)\textup{.}
\end{enumerate}
Furthermore\textup{,} assume that \(\ms{M}_{0}\) is generated by its global sections\textup{.}
\begin{enumerate}[resume*]
\item The sheaf \(\ms{M}\) is generated by its global sections\textup{.}
\item The map \(\bar{\alpha}\co f^{\tri}M\ra \ms{M}\) induced by \(\alpha\) is an isomorphism\textup{.}
\end{enumerate}
\end{lem}
\begin{proof}
(i) Note that \(f_{0}\) induces an isomorphism \(\Spec ((f_{0})_{*}\Q_{Y_{0}})\ra X_{0}\) since \(X_{0}\) is assumed to be normal and \(f_{0}\) birational; \cite[4.3.12]{EGAIII1}. Since $\mr{R}^2(f_0)_*(-)=0$, by Corollary \ref{cor.extbc}, $\mr{R}^{1}\hspace{-0.15em}f_{*}\ms{M}$ commutes with base change, and is $S$-flat by assumption. By Proposition \ref{prop.extbc} (ii) with $n=1$, $f_{*}\ms{M}$ commutes with base change, and is $S$-flat by Proposition \ref{prop.extbc} (ii) with $n=0$. In particular this holds for $\ms{M}=\Q_Y$. Hence the pair \((\Spec f_{*}\Q_{Y},f_{*}\ms{M})\) is a deformation of \((X_{0},M_{0})\), isomorphic to \((X,M)\) through the maps \(\Spec f^{\sharp}\) and \(\alpha^{\tn{ad}}\). 

(ii-iii) Since \(f_{*}\ms{M}\) is a deformation of \((f_{0})_{*}\ms{M}_{0}\), global sections generating \(\ms{M}_{0}\) lift to global sections generating \(\ms{M}\). Hence the map \(\bar{\alpha}\) is surjective; cf.\ \eqref{eq.glob}. But \(\bar{\alpha}\) is injective too since the strict transform here commutes with base change by Lemma \ref{lem.U} and Corollary \ref{cor.blowup},
and \(\ms{M}\) is \(S\)-flat.
\end{proof}
%%%%%%%%%%%%%%%%%%% EKSEMPEL %%%%%%%%%%%%%%%%%%%%%%%%%%%%
\begin{ex}\label{ex.fib}
Let \(X_{0}\) be a rational surface singularity, \(M_{0}\) a reflexive module and \(\ms{M}_{0}=f_{0}^{\tri}M_{0}\) (the topical case). Then \(\mr{R}^{1}\hspace{-0.15em}f_{*}\Q_{Y}=0\) follows from \(\mr{R}^{1}\hspace{-0.08em}(f_{0})_{*}\Q_{Y_{0}}=0\) by Proposition \ref{prop.extbc} (i) and Nakayama's lemma.
Since \(\ms{M}_{0}\) is generated by its global sections, \(\mr{R}^{1}\hspace{-0.08em}(f_{0})_{*}\ms{M}_{0}=0\) which implies \(\mr{R}^{1}\hspace{-0.15em}f_{*}\ms{M}=0\) again by Proposition \ref{prop.extbc} (i) and Nakayama's lemma. 
\end{ex}
%%%%%%%%%%%%%%%%%%% KOROLLAR %%%%%%%%%%%%%%%%%%%%%%%%%%%%
\begin{cor}
Identifying isomorphic objects in the fibres of  \(\cdf{}{(Y_{0}/X_{0},\ms{M}_{0}/M_{0})}/\cat{H}_{k}\) defines a deformation functor \(\df{}{(Y_{0}/X_{0},\ms{M}_{0}/M_{0})}\co \cat{H}_{k}\ra \Sets\)\textup{.} 
There are correspondingly defined fibred categories and deformation functors \(\df{}{(X_{0},M_{0})}\)\textup{,} \(\df{}{Y_{0}/X_{0}}\) and \(\df{}{X_{0}}\)\textup{.} Moreover\textup{,} there is a commutative diagram of forgetful maps\textup{:} 
\begin{equation*}
\xymatrix@C+6pt@R-6pt@H-6pt{
\hspace{0em}\df{}{(Y_{0}/X_{0},\ms{M}_{0}/M_{0})} \ar[r]\ar@<-2.8em>[d] & \df{}{Y_{0}/X_{0}}\hspace{-1.7em}\ar[d]
\\
\hspace{-3em}\df{}{(X_{0},M_{0})} \ar[r] & \df{}{X_{0}}\hspace{-0.5em}
}
\end{equation*}
\end{cor}
We write `local family' to indicate membership in any of the fibred categories.
%%%%%%%%%%%%%%%%%%% DEFINISJON %%%%%%%%%%%%%%%%%%%%%%%%%%%%
\begin{defn}\label{def.smooth}
Suppose \(F\) and \(G\) are set-valued contravariant functors of \(\cat{H}_{k}\) with \(|F(\Spec k)|=|G(\Spec k)|=1\). Then a natural transformation \(\phi\co F\ra G\) is \emph{smooth} if for all closed embeddings \(S\ra R\) in \(\cat{H}_{k}\), the natural map \(F(R)\ra F(S)\times_{G(S)}G(R)\) is surjective.
\end{defn}
In particular \(\phi\) is surjective. With this definition versality of a pair \((R,\xi)\), \(\xi\in F(R)\) is the same as smoothness of the corresponding Yoneda map \(h_{R}\ra F\) and \(R\) algebraic over \(k\).
%%%%%%%%%%%%%%%%%%%%%%%%%%%%%%%%%%%%%%%%%%%%%%%%%%%%%%
%%%%%%%%%%%%%% SECTION %%%%%%%%%%%%%%%%%%%%%%%%%%%%%%%
%%%%%%%%%%%%%%%%%%%%%%%%%%%%%%%%%%%%%%%%%%%%%%%%%%%%%%
\section{Normality and McKay-Wunram correspondence of blowing-up}\label{sec.MW}
We prove normality of the blowing-up of a rational surface singularity in a reflexive module and a McKay-Wunram correspondence with such blowing-ups.
The following statement about deformations is a key ingredient in the proofs of the main results.
%%%%%%%%%%%%%% LEMMA %%%%%%%%%%%%%%%%%%%%%%%%%%%%%%%%%
\begin{lem}\label{lem.MJ}
Suppose \(f_{0}\co Y_{0}\ra X_{0}\) is a partial resolution of a rational surface singularity and \(M_{0}\) a rank \(r\) reflexive \(\Q_{X_{0}}\)-module\textup{.} Assume that \(\ms{M}_{0}=f_{0}^{\tri}M_{0}\) is locally free on \(Y_{0}\)\textup{.} Let \((f\co Y\ra X,\ms{M}/M)\) be a deformation in \(\df{}{(Y_{0}/X_{0},\ms{M}_{0}/M_{0})}(S)\)\textup{.} 
Let \(0\ra\Q_{X}^{\oplus r-1}\ra M\ra J\ra 0\) be the short exact sequence of \(\Q_{X}\)-modules defined by a lifting of \(r-1\) elements in \(M_{0}\) with the property in \textup{Lemma \ref{lem.chern}}\textup{.} 
\begin{enumerate}[leftmargin=2.4em, label=\textup{(\roman*)}]
\item There are natural
isomorphisms of \(\Q_{X}\)-modules\textup{:} 
\begin{equation*}
\llb M\rrb \cong f_{*}\textstyle\bigwedge^{\hspace{-0.12em}r}\hspace{-0.2em}\ms{M}\cong J
\end{equation*}
\item Put \(\ms{F}_{0}=\bigwedge^{\hspace{-0.12em}r}\hspace{-0.2em}\ms{M}_{0}\)\textup{.} Then the map 
\begin{equation*}
\textstyle
\eta\co\df{}{(Y_{0}/X_{0},\hspace{0.1em}\ms{M}_{0}/M_{0})}
\lra
\df{}{(Y_{0}/X_{0},\hspace{0.18em}
\ms{F}_{0}/f_{0}{}_{*}\ms{F}_{0})}
\end{equation*}
given by \((f\co Y\ra X,\ms{M}/M)\mapsto (f\co Y\ra X,\bigwedge^{\hspace{-0.12em}r}\hspace{-0.3em}\ms{M}/f_{*}\bigwedge^{\hspace{-0.12em}r}\hspace{-0.3em}\ms{M})\) is well defined and
smooth\textup{.}
\end{enumerate}
\end{lem}
\begin{proof}
(i) There is a natural map \(e\co\bigwedge^{\hspace{-0.12em}r} \hspace{-0.1em}f_{*}\ms{M}\ra f_{*}\bigwedge^{\hspace{-0.12em}r}\hspace{-0.2em}\ms{M}\). We prove that \(e\) is surjective. Let \(s_{2},\dots,s_{r}\) denote the given global sections in \(\ms{M}\) and let \(\ms{L}\) denote the cokernel of the induced map \(i\co \Q_{Y}^{\oplus r-1}\ra\ms{M}\).
The central fibre \(i_{0}\) of \(i\) is injective and \(\coker i_{0}\) is an invertible sheaf (Lemma \ref{lem.chern}). It follows that \(\ms{L}\) is \(S\)-flat and invertible. 
Since \(\ms{M}\ra\ms{L}\) is locally split, the \(\Q_{Y}\)-linear map \(w\co \ms{M}\ra\bigwedge^{\hspace{-0.12em}r}\hspace{-0.2em}\ms{M}\) defined by \(m\mapsto m\wedge s_{2}\wedge\dots\wedge s_{r}\) is surjective and the induced sequence
\begin{equation}\label{eq.chern}
\xi\co\quad 0\ra \Q_{Y}^{\oplus r-1}\xra{\hspace{0.4em}i\hspace{0.4em}} \ms{M}\xra{\,\,w\,\,}\textstyle\bigwedge^{\hspace{-0.12em}r}\hspace{-0.2em}\ms{M}\ra 0 
\end{equation}
is short exact. Push forward of \(\xi\) by \(f\) gives a short exact sequence 
\begin{equation}\label{eq.dws}
0\ra\Q_{X}^{\oplus r-1}\xra{\hspace{0.2em}f_{*}i\hspace{0.2em}} M\xra{\hspace{0.2em}f_{*}w\hspace{0.2em}} f_{*}\textstyle\bigwedge^{\hspace{-0.12em}r}\hspace{-0.2em}\ms{M}\ra 0
\end{equation}
by Lemma \ref{lem.fib} and Example \ref{ex.fib}.
Note that \(f_{*}w\) factors via \(e\) and \(e\) is thus surjective.
By Lemma \ref{lem.U}, \(e\) is generically injective. Since \(f_{*}\bigwedge^{\hspace{-0.12em}r}\hspace{-0.2em}\ms{M}\) is torsion free, \(e\) induces a map \(\bar{e}\co \llb M\rrb\ra f_{*}\bigwedge^{\hspace{-0.12em}r}\hspace{-0.2em}\ms{M}\cong J\) which is an isomorphism.

(ii) The sequences \eqref{eq.chern} and \eqref{eq.dws} implies that \((f\co Y\ra X,\bigwedge^{\hspace{-0.12em}r}\hspace{-0.3em}\ms{M}/f_{*}\bigwedge^{\hspace{-0.12em}r}\hspace{-0.3em}\ms{M})\) is a deformation. Since base change of \eqref{eq.chern} and \eqref{eq.dws} give sequences with the same properties, the map \(\eta\) is well defined. 
 
Suppose \(i\co S\ra R\) is a closed immersion in \(\cat{H}_{k}\) and \((f'\co Y'\ra X',\ms{L}/f'_{*}\ms{L})\) an element in \(\df{}{(Y_{0}/X_{0},\hspace{0.18em}\ms{F}_{0}/f_{0}{}_{*}\ms{F}_{0})}(R)\) which restricts to \((Y/X,\bigwedge^{\hspace{-0.12em}r}\hspace{-0.2em}\ms{M}/f_{*}\hspace{-0.1em}\bigwedge^{\hspace{-0.12em}r}\hspace{-0.2em}\ms{M})\).  
Since \(\xt{2}{Y_{0}}{\bigwedge^{\hspace{-0.12em}r}\hspace{-0.2em}\ms{M}_{0}}{\Q_{Y_{0}}^{\oplus r-1}}=0\), Corollary \ref{cor.extbc} implies that the base change map 
\begin{equation}
\xt{1}{Y'}{\ms{L}}{\Q_{Y'}^{\oplus r-1}}\lra \xt{1}{Y}{{\textstyle\bigwedge}^{\hspace{-0.12em}r}\hspace{-0.2em}\ms{M}}{\Q_{Y}^{\oplus r-1}}
\end{equation}
is surjective. In particular there is a short exact sequence 
\begin{equation}
\xi'\co\hspace{0.8em} 0\ra\Q_{Y'}^{\oplus r-1}\lra \ms{E}\lra\ms{L}\ra 0 
\end{equation}
of \(\Q_{Y'}\)-modules which pulls back to \(\xi\). Then \(\ms{E}\) is locally free and \(R\)-flat. Moreover, \(f'_{*}\ms{E}\) is a deformation of \(M\) by Corollary \ref{cor.extbc}.  Then \((Y'/X',\ms{E}/f'_{*}\ms{E})\) is a deformation of \((Y/X,\ms{M}/M)\) and as in (i) we have \(\bigwedge^{\hspace{-0.12em}r}\hspace{-0.1em}\ms{E}\cong\ms{L}\) above \(\textstyle\bigwedge^{\hspace{-0.12em}r}\hspace{-0.2em}\ms{M}\). 
\end{proof}
%%%%%%%%%%%%%%%%%%%%%%%%%%%%%%%%%%%%%%%%%%%%%%%%%%%%%%
\begin{prop}[Normality]\label{prop.normal}
Let \(X\) be a surface with only rational singularities and \(\pi\co Y\ra X\) the blowing-up of \(X\) in a reflexive \(\Q_{X}\)-module \(M\)\textup{.} Then \(Y\) is normal\textup{.} 
\end{prop}
\begin{proof}
We may assume \(X\) is a rational surface singularity.  
The strict transform \(\ms{F}=\pi^{\tri}M\) along the minimal resolution \(\pi\co \tilde{X}\ra X\) is locally free by Proposition \ref{prop.lCM}. By Lemma \ref{lem.MJ}, \(\llb M\rrb\cong \pi_{*}(\bigwedge^{\hspace{-0.12em}r}\hspace{-0.2em}\ms{F})
\). Since \(\bigwedge^{\hspace{-0.12em}r}\hspace{-0.2em}\ms{F}\) is an invertible sheaf, \(\llb M\rrb\) is an integrally closed fractional ideal by \cite[5.3]{lip:69}. By \cite[8.1]{lip:69} the blowing-up \(\Bl_{\llb M\rrb}(X)\cong Y\) is normal.
\end{proof}
The following class of reflexive modules was introduced in \cite{wun:88}.
Let \(\pi\co \tilde{X}\ra X\) be the minimal resolution of a rational surface singularity, \(M\) a (non-trivial) reflexive \(\Q_{X}\)-module and \(\ms{F}=\pi^{\tri}M\) the strict transform. Put \(\ms{F}^{\omega}=\shm{}{\tilde{X}}{\ms{F}}{\omega_{\tilde{X}}}\) and \(\ms{F}^{\vee}=\shm{}{\tilde{X}}{\ms{F}}{\Q_{\tilde{X}}}\). While in general \(\mr{R}^{1}\pi_{*}\ms{F}^{\omega}=0\) (Proposition \ref{prop.lCM}), we say that \(M\) is \emph{Wunram} if the stronger condition \(\mr{R}^{1}\pi_{*}\ms{F}^{\vee}=0\) holds. Note that for RDPs all reflexive are Wunram since $\omega_{\tilde{X}}\cong \Q_{\tilde{X}}$. Wunram constructed the indecomposabel non-projective Wunram modules as follows. Let $D_i$ be an effective prime divisor transversal to the prime component $E_i$ in the fundamental cycle $E(\pi)$ as in Proposition \ref{prop.pic} (iii). Choose a minimal number of $r_i$ generating global sections in $\Q_{D_i}$. Let $\ms{G}$ be the kernel of the induced map $\Q_{\tilde{X}}^{\oplus r_i}\ra \Q_{D_i}$. Then $\ms{G}$ is locally free of rank $r_i$. Put $\ms{F}_i=\ms{G}^{\vee}$ and $M_i=\pi_*\ms{F}_i$. One obtains sequences $\alpha$ and $\beta$ as in Lemma \ref{lem.chern}. Applying $\hm{}{\tilde{X}}{-}{\Q_{\tilde{X}}}$ to $\alpha$ gives a short exact sequence on $X$ by choice and $\mr{R}^{1}\pi_{*}\ms{F}_i^{\vee}=0$. Then $M_i$ is reflexive and $\pi^{\tri}M_i\cong \ms{F}_i$; cf. Proposition \ref{prop.lCM}. Moreover, $r_i=\dim_k\cH^0(\Q_{D_i})=\mr{c}_{1}(\ms{F}_i).E(\pi)$ which equals the multiplicity of $E_i$ in the fundamental cycle $E(\pi)$. This is Wunram's direct generalisation of the (geometric) McKay correspondence (cf.\ \cite{gon-spr/ver:83}, \cite[1.11]{art/ver:85}); see \cite[1.2]{wun:88} which also contains a `multiplication formula'. Note that Wunram's result is stated in the analytic category, but his proof of \cite[1.2]{wun:88} holds in all characteristics (with henselian local rings). Iyama and Wemyss generalised Wunram modules to all normal surface singularities with several characterisations in \cite[2.6-7]{iya/wem:10}. Van den Bergh gave a higher dimensional generalisation (of the sheaves) in \cite[3.5.1-4]{vdber:04}.

We prove a blowing-up version of the McKay-Wunram correspondence.
%%%%%%%%%%%%%%%%%%%%%%%%%%%%%%%%%%%%%%%%%%%%%%%%%%%%%%
\begin{thm}\label{thm.MW}
Let \(\pi\co\tilde{X}\ra X\) be the minimal resolution of a rational surface singularity\textup{.} 
\begin{enumerate}[leftmargin=2.4em, label=\textup{(\roman*)}]
\item Blowing \(X\) up in a reflexive \(\Q_{X}\)-module \(M\) gives a partial resolution \(f\co Y=\Bl_{M}(X)\ra X\) dominated by the minimal resolution\textup{.} The partial resolution is obtained by contracting the prime components \(\{E_{i}\,\vert\,\ch_{1}(\pi^{\tri}M).E_{i}=0\}\) of the exceptional divisor \(E(\pi)\) in \(\tilde{X}\)\textup{.} 
\item Every partial resolution of \(X\) dominated by the minimal resolution is given by blowing up \(X\) in a Wunram \(\Q_{X}\)-module \(M\) and two Wunram modules give isomorphic partial resolutions if and only if they have the same non-free indecomposable summands\textup{.} 
\item The association \(M'\mapsto \ch_{1}(f^{\tri}M')\) gives a one-to-one correspondence between stable isomorphism classes of Wunram modules with the same non-free indecomposable summands as \(M\)\textup{,} and isomorphism classes of ample invertible sheaves on \(Y\)\textup{.}
\item If \(M=M_{i}\) is an indecomposable Wunram module then \(E(f)_{\tn{red}}\cong \BB{P}^{1}\) is the image of \(E_{i}\) under the contraction map \(\tilde{X}\ra Y\)\textup{.} The rank of \(M_{i}\) equals the \textup{(}generic\textup{)} multiplicity of the unreduced exceptional fiber \(E(f)\) in Y\textup{.} 
\end{enumerate}
\end{thm}
%%%%%%%%%%%%%%%%%%%%%%%%%%%%%%%%%%%%%%%%%%%%%%%%%%%%%%
\begin{proof}
(i) Proposition \ref{prop.normal} gives normality of \(Y\). By Proposition \ref{prop.lCM} and the universality of the blowing-up in Proposition \ref{prop.blowup}, \(\pi\) factors through \(f\).
Suppose \(g\co Y'\ra X\) is a partial resolution dominated by the minimal resolution such that \(\ms{M}=g^{\tri}M\) is locally free. Let \(\bar{E}_{j}\subset Y'\) denote the image of some exceptional component \(E_{j}\subset \tilde{X}\). Let \(h\co Y'\ra Y^{\dprime}\) be the contraction of \(\bar{E}_{j}\) with \(y=h(\bar{E}_{j})\). Put \(Y^{\dprime}_{y}=\Spec\Q_{Y^{\dprime}\hspace{-0.2em},\hspace{0.1em}y}^{\tn{h}}\) and let \(p\co V\ra Y^{\dprime}_{y}\) be the base change of \(h\) along the natural \(l_{y}\co Y^{\dprime}_{y}\ra Y^{\dprime}\). Then \(p\) is a resolution of a rational surface singularity. Let \(q\co V\ra Y'\) denote the projection and \(g'\co Y^{\dprime}\ra X\) the natural map with \(g=g'h\). Then \((g')^{\tri}M\cong h_{*}\ms{M}\) is locally free \(\lRa\) \(p^{\tri}l_{y}^{*}(g')^{\tri}M\cong\Q_{V}^{\oplus \rk M}\) by Lemma \ref{lem.chern}. Since \(h^{\tri}h_{*}\ms{M}\cong\ms{M}\) by Proposition \ref{prop.lCM}, Corollary \ref{cor.blowup} gives the isomorphism \(q^{*}\ms{M}\cong p^{\tri}l_{y}^{*}h_{*}\ms{M}\). It follows by Lemma \ref{lem.chern} and Proposition \ref{prop.pic} (iia) that \(h_{*}\ms{M}\) is locally free \(\lRa\) \(\ch_{1}(q^{*}\hspace{-0.2em}\ms{M}).q^{-1}(\bar{E}_{j})=0\). Finally \(\ch_{1}(q^{*}\hspace{-0.2em}\ms{M}).q^{-1}(\bar{E}_{j})=\ch_{1}(\ms{M}).\bar{E}_{j}\).

(ii-iii) are direct consequences of (i), Wunram's \cite[1.2]{wun:88} and Proposition \ref{prop.pic}.

(iv) See Proposition \ref{prop.pic}. The generic multiplicity of \(E(f)\) equals
\(\ch_{1}(\pi^{\tri}M_{i}).E(\pi)\) which by Wunram's \cite[1.2]{wun:88} equals \(\rk M_{i}\).
\end{proof}
Since all reflexive modules are Wunram if \(X\) is an RDP we retain Curto and Morrisons theorem \cite[2.2]{cur/mor:13}, however with the strengthening that the blowing-up of \(X\) in a reflexive module is normal. For a very different construction of minimal (and partial) resolutions of rational singularities employing the Wunram modules, see \cite[5.4.2]{kar:17}.
%%%%%%%%%%%%%% EKSEMPEL %%%%%%%%%%%%%%%%
\begin{ex}\label{ex.omega}
Let \(f\co \tilde{X}^{\tn{c}}\ra X\) be the partial resolution obtained by contracting the \((-2)\)-curves in \(\tilde{X}\). Then \(f\) is called the \emph{RDP-resolution} of \(X\). In particular, \(\tilde{X}^{\tn{c}}\) has only RDP-singularities and is the canonical model of \(X\). By rationality \(\pi_{*}\omega_{\tilde{X}}\cong \omega_{X}\); \cite[4.12]{bad:01}. By Proposition \ref{prop.lCM}, \(\pi^{\tri}\omega_{X}\cong\omega_{\tilde{X}}\). For any \(E_{i}\), adjunction gives \(\omega_{\tilde{X}}.E_{i} = -2-E_{i}^{2}\), hence Theorem \ref{thm.MW} implies that the RDP-resolution is given by blowing up \(X\) in \(\omega_{X}\). If \(i\co U\ra X\) denotes the regular locus, let \(\omega_{X}^{n}\) denote the image of the natural map \(\omega_{X}^{\otimes n}\ra i_{*}(\omega_{U}^{\ot n})\).  Then \(\tilde{X}^{\tn{c}}\cong\Bl_{\omega_{X}}(X)\cong\Proj(\bigoplus_{n\geq 0}\omega_{X}^{n})\) which is the scheme-theoretic closed image of \(U\) in \(\BB{P}(\omega_{X})\); cf. \eqref{eq.rk1}.
\end{ex}
%%%%%%%%%%%%%%%%%%%%%%%%%%%%%%%%%%%%%%%%%%%%%%%%%%%%%%
%%%%%%%%%%%%%% SECTION %%%%%%%%%%%%%%%%%%%%%%%%%%%%%%%
%%%%%%%%%%%%%%%%%%%%%%%%%%%%%%%%%%%%%%%%%%%%%%%%%%%%%%
\section{The main theorem}\label{sec.square} 
%%%%%%%%%%%%%%% THEOREM %%%%%%%%%%%%%%%%%%%%%%%%%%%%%%
\begin{thm}\label{thm.square}
Let \(f\co Y\ra X\) be the blowing-up of a rational surface singularity in a reflexive \(\Q_{X}\)-module \(M\)\textup{.} Let \(\ms{M}\) denote the strict transform  \(f^{\tri}\hspace{-0.09em}M\)\textup{.} The forgetful maps give a commutative diagram of deformation functors
\begin{equation*}\label{eq.mainsq1}
\xymatrix@C+12pt@R-6pt@H-6pt{
\hspace{-0.9em}\df{}{(Y/X,\ms{M}/M)} \ar[r]^(0.6){\beta}\ar@<-2.5em>[d]_(0.45){\alpha} & \df{}{Y/X}\hspace{-1.6em}\ar@<0.12em>[d]
\\
\hspace{-3em}\df{}{(X,M)} \ar[r] & \df{}{X}\hspace{-0.5em}
}
\end{equation*}
with the following properties\textup{:}
\begin{enumerate}[leftmargin=2.4em, label=\textup{(\roman*)}]
\item \(\alpha\) is injective
\item \(\beta\) is smooth\textup{,} in particular surjective
\item \(\beta\) is an isomorphism if \(\ms{M}\) is rigid\textup{.} In particular\textup{,} \(\beta\) is an isomorphism if \(M\) is a Wunram module or if \(\rk M=1\)\textup{.}
\end{enumerate}
\end{thm}
The proof of Theorem \ref{thm.square} is divided into several steps. The following result implies (i) and the stronger statement will be needed in the application to flops.
%%%%%%%%%%%%%% PROPOSITION %%%%%%%%%%%%%%%%%%%%%%%%%%%
\begin{prop}\label{prop.blowingup}
Let \(f_{0}\co Y_{0}\ra X_{0}\) be the blowing-up of a rational surface singularity in a reflexive \(\Q_{X_{0}}\)-module \(M_{0}\)\textup{.} Let \(\ms{M}_{0}\) denote the strict transform  \(f^{\tri}M_{0}\)\textup{.} Put 
\begin{equation*}
\df{\dprime}{(X_{0},M_{0})}=\im\{\alpha\co \df{}{(Y_{0}/X_{0},\ms{M}_{0}/M_{0})}\ra \df{}{(X_{0},M_{0})}\}
\end{equation*}
Then blowing-up gives a map \(\gamma \co\df{\dprime}{(X_{0},M_{0})}\ra\df{}{(Y_{0}/X_{0},\ms{M}_{0}/M_{0})}\) such that the composition
\begin{equation*}
\df{}{(Y_{0}/X_{0},\ms{M}_{0}/M_{0})}\xra{\hspace{0.3em}\alpha\hspace{0.3em}}\df{\dprime}{(X_{0},M_{0})}\xra{\hspace{0.3em}\gamma\hspace{0.3em}}\df{}{(Y_{0}/X_{0},\ms{M}_{0}/M_{0})}
\end{equation*}
is the identity\textup{.}
\end{prop}
\begin{proof}
Let \((f\co Y\ra X,\ms{M}/M)\) be an element in \(\df{}{(Y_{0}/X_{0},\ms{M}_{0}/M_{0})}(S)\). Let \(f'\co Y'\ra X\) denote the blowing-up of \(X\) in \(M\) with \((f')^{*}M\ra (f')^{\tri}M=\ms{M}'\) the quotient map of sheaves. It gives a map of pairs \((Y'/X,\ms{M}'/M)\) which is a deformation of \((Y_{0}/X_{0},\ms{M}_{0}/M_{0})\) by Lemma \ref{lem.U} and Corollary \ref{cor.blowup}.  
By the universal property in Proposition \ref{prop.blowup} there is a unique factorisation \(g\co Y\ra Y'\) of \(f\) with \(g^{*}\ms{M}'\cong \ms{M}\). The restriction of \(g\) to the central fibre is an isomorphism. It follows that \(g\) is an isomomorphism (cf. the proof of Lemma \ref{lem.U}) which implies that \((f,\ms{M}/M)\cong (f',\ms{M}'/M)\) as deformations. Hence \(\gamma\) is well defined with \(\gamma\alpha\simeq\id\).
\end{proof}
%%%%%%%%%%%% LEMMA %%%%%%%%%%%%%%%%%%%%%%%%%%%%%%%%%%%%%%%
\begin{lem}\label{lem.inj}
Let \(f_{0}\co Y_{0}\ra X_{0}\) be a partial resolution of a normal surface singularity and \(\ms{M}_{0}\) a locally free\textup{,} coherent \(\Q_{Y_{0}}\)-module\textup{.} Put \(M_{0}=(f_{0})_{*}\ms{M}_{0}\)\textup{.} Assume \(\xt{1}{Y_{0}}{\ms{M}_{0}}{\ms{M}_{0}}=0\)\textup{.}
Then the forgetful map \(\df{}{(Y_{0}/X_{0},\ms{M}_{0}/M_{0})}\ra\df{}{Y_{0}/X_{0}}\) is injective\textup{.}
\end{lem}
\begin{proof}
Given elements \((Y/X,\ms{M}/M)\) and \((Y'/X',\ms{M}'/M')\) in \(\df{}{(Y_{0}/X_{0},\ms{M}_{0}/M_{0})}(S)\) such that \((f\co Y\ra X)\cong (f'\co Y'\ra X')\) as deformations of \(f_{0}\). We use this isomorphism to identify \(f'\) with \(f\). Corollary \ref{cor.extbc} gives that the base change map of \(\mr{H}^{0}(\Q_{Y})\)-modules \(\hm{}{Y}{\ms{M}}{\ms{M}'}\ra\nd{}{Y_{0}}{\ms{M}_{0}}\) is a deformation. In particular there is an \(\Q_{Y}\)-linear homomorphism \(\theta\co \ms{M}\ra\ms{M}'\) lifting \(\id_{\ms{M}_{0}}\). It follows that \(\theta\) is an isomorphism of deformations (since \(\ms{M}'\) is \(S\)-flat) which, pushed down, gives an isomorphism \(M\cong M'\).
\end{proof}
%%%%%%%%%%%% LEMMA %%%%%%%%%%%%%%%%%%%%%%%%%%%%%%%%%%%%%%
\begin{lem}\label{lem.piciso}
Let \(f_{0}\co Y_{0}\ra X_{0}\) be a partial resolution of a rational surface singularity and \(\ms{L}_{0}\) an invertible \(\Q_{Y_{0}}\)-module generated by its global sections\textup{.} 
Put \(M_{0}=(f_{0})_{*}\ms{L}_{0}\)\textup{.} Then the forgetful map 
$
\df{}{(Y_{0}/X_{0},\ms{L}_{0}/M_{0})}\lra\df{}{Y_{0}/X_{0}}
$
is an isomorphism\textup{.}
\end{lem}
\begin{proof} 
By Lemma \ref{lem.inj} we only have to show surjectivity. Let \(f\co Y\ra X\) be an element in \(\df{}{Y_{0}/X_{0}}(S)\). Proposition \ref{prop.pic} implies that there is an effective Cartier divisor \(D_{0}\) in \(Y_{0}\) intersecting \(E(f)_{\tn{red}}\) transversally with \(\ms{L}_{0}\cong\Q_{Y_{0}}(D_{0})\). Locally around \(\Supp D_{0}\) there is a non-zero-divisor \(t_{0}\) defining \(D_{0}\). Any local section \(t\) in \(\Q_{Y}\) lifting \(t_{0}\) is a non-zero-divisor and defines an \(S\)-flat divisor \(D\) in \(Y\). Put \(M=f_{*}\Q_{Y}(D)\). Then \((f,\Q_{Y}(D)/M)\) is a deformation of \((f_{0},\ms{L}_{0}/M_{0})\).
\end{proof}
%%%%%%%%%%%%%%%%%%%%%%%%%%%%%%%%%%%%%%%%%%%%%%%%%%
\begin{proof}[Proof of Theorem {\ref{thm.square}}]
Proposition \ref{prop.blowingup} implies injectivity of \(\alpha\). By Lemma \ref{lem.MJ} and Lemma \ref{lem.piciso}, \(\beta\) is smooth, and with Lemma \ref{lem.inj} an isomorphism if \(f^{\tri}M\) is rigid.

For the Wunram case, let \(g\co \tilde{Y}\ra Y\) be the minimal resolution. Put \(\pi=fg\) and
\(\ms{F}=\pi^{\tri}M\). Then \(g^{*}\ms{M}\cong\ms{F}\) and \(g_{*}\ms{F}\cong \ms{M}\) by Proposition \ref{prop.lCM}.
By Theorem \ref{thm.MW}, \(Y\) is normal so \(g_{*}\Q_{\tilde{Y}}\cong \Q_{Y}\). The natural isomorphism \(\ms{M}^{\vee}\cong g_{*}(\ms{F}^{\vee})\) follows. 
The Leray spectral sequence gives a short exact sequence
\begin{equation}
0\ra \mr{R}^{1}\hspace{-0.12em}f_{*}(g_{*}(\ms{F}^{\vee}))\lra \mr{R}^{1}\pi_{*}(\ms{F}^{\vee})\lra f_{*}\mr{R}^{1}g_{*}(\ms{F}^{\vee})\ra 0.
\end{equation}
If \(M\) is Wunram then \({\mr{R}}^{1}\pi_{*}(\ms{F}^{\vee})=0\) and so \(\mr{R}^{1}\hspace{-0.12em}f_{*}(\ms{M}^{\vee})=0\). As \(\ms{M}\) is generated by its global sections there is a surjection \(\Q_{Y}^{\oplus n}\ra\ms{M}\). It induces a surjection 
\begin{equation}
\cH^{1}(Y,\shm{}{\Q_{Y}}{\ms{M}}{\Q_{Y}^{\oplus n}})\lra\cH^{1}(Y,\snd{}{\Q_{Y}}{\ms{M}}).
\end{equation}
Hence \(\ms{M}\) is rigid.
\end{proof}
%%%%%%%%%%%%%%% REMARK %%%%%%%%%%%%%%%%%%%%%%%%%%%
\begin{rem}
The following result is a corollary of Theorem \ref{thm.square}. The proof shows that \(\df{\dprime}{(X_{0},M_{0})}\) in Proposition \ref{prop.blowingup} is the largest subfunctor of \(\df{}{(X_{0},M_{0})}\) for which blowing up gives a flat family.
\end{rem}
%%%%%%%%%%%%%%% COROLLARY %%%%%%%%%%%%%%%%%%%%%%%%%%%%%%
\begin{cor}\label{cor.square}
Let \(f_{0}\co Y_{0}\ra X_{0}\) be the blowing-up of a rational surface singularity in a reflexive \(\Q_{X_{0}}\)-module \(M_{0}\)\textup{.} Put \(\ms{M}_{0}=f_{0}^{\tri}M_{0}\). 
Then the functors \(\df{}{(X_{0},M_{0})}\)\textup{,} \(\df{}{(Y_{0}/X_{0},\hspace{0.1em}\ms{M}_{0}/M_{0})}\) and \(\df{}{Y_{0}/X_{0}}\) all have versal elements\textup{.} 
\end{cor}
\begin{proof}
By \cite[10.2]{ile:11x} the functor \(\df{}{(X_{0},M_{0})}\) has a versal element, say \((X,M)\in \df{}{(X_{0},M_{0})}(R)\). Let \(f\co Y=\Bl_{M}(X)\ra X\) denote the blowing-up. By Proposition \ref{prop.blowup}, Lemma \ref{lem.U} and Corollary \ref{cor.blowup} the closed fibre equals \(f_{0}\).
By choosing a finite type representative, \cite[\href{https://stacks.math.columbia.edu/tag/05PI}{Lemma 05PI}]{SP} gives a flattening subscheme \(\bar{R}\sbeq R\) for \(Y\ra R\). Let \((\bar{X},\bar{M})\) denote the induced image in \(\df{}{(X_{0},M_{0})}(\bar{R})\) and \(\bar{f}\co \bar{Y}\ra \bar{X}\) the pullback of \(f\). Put \(\ms{M}=f^{\tri}M\). There is a natural map \(\bar{f}^{*}\bar{M}\ra \ms{M}_{|\bar{Y}}=:\bar{\ms{M}}\) and \((\bar{Y}/\bar{X},\bar{\ms{M}}/\bar{M})\) is an element in \(\df{}{(Y_{0}/X_{0},\ms{M}_{0}/M_{0})}(\bar{R})\). 

To test for versality of \((\bar{Y}/\bar{X},\bar{\ms{M}}/\bar{M})\) apply versality of \((X,M)\) and the universality of \(\bar{Y}\ra \bar{R}\). Versality follows since the forgetful map \(\alpha\) in Theorem \ref{thm.square} is injective. 

Moreover, since \(\beta\) in Theorem \ref{thm.square} is smooth, \(\bar{f}\co\bar{Y}\ra \bar{X}\) is a versal element in \(\df{}{Y_{0}/X_{0}}(\bar{V})\). 
\end{proof}
%%%%%%%%%%%%%%% COROLLARY %%%%%%%%%%%%%%%%%%%%%%%%%%%%%%
\begin{cor}[Lipman {\cite{lip:79}}]\label{cor.Lip}
Let \(X_{0}\) be a rational surface singularity and let \(f_{0}\co \tilde{X}_{0}^{\tn{c}}\ra X_{0}\) denote the RDP-resolution\textup{.} Then the forgetful map \(\df{}{\tilde{X}_{0}^{\tn{c}}/X_{0}}\ra \df{}{X_{0}}\) is injective\textup{.}
\end{cor}
\begin{proof}
Note that \(\tilde{X}_{0}^{\tn{c}}\cong \Bl_{\omega_{X_{0}}}(X_{0})\); see Example \ref{ex.omega}. For a deformation \(X/S\) let \(\omega_{X/S}\) denote (the henselisation of) the dualising module. It is \(S\)-flat with canonical modules in the fibres; cf. \cite[Section 3.5]{con:00}. The map \(\df{}{X_{0}}\ra \df{}{(X_{0},\omega_{X_{0}})}\) defined by \(X/S\mapsto (X/S,\omega_{X/S})\) is an isomorphism (use Corollary \ref{cor.extbc} as in the proof of Lemma \ref{lem.inj}). By Theorem \ref{thm.square} the result follows.
\end{proof}
%%%%%%%%%%%%% REMARK %%%%%%%%%%%%%%
\begin{rem}\label{rem.A}
Let \(X/S\) be the minimal versal element in \(\df{}{X_{0}}\). Consider the functor \(\Res_{X/S}\) of local henselian schemes over \(S\) where \(\Res_{X/S}(S'/S)\) is the set of (isomorphism classes of) proper maps \(Y\ra X_{S'}\) such that \(Y\) is \(S'\)-flat and the closed fibre \(Y_{0}\ra X_{0}\) is the minimal resolution. There is a choice of finite type representative \(X^{\tn{ft}}\hspace{-0.1em}/S^{\tn{ft}}\hspace{-0.1em}\) of \(X/S\) with finite singular locus over \(S^{\tn{ft}}\hspace{-0.1em}\) such that \(\Res_{X/S}\) is represented by the henselisation in \(X_{0}^{\tn{ft}}\hspace{-0.1em}/\Spec k\)  of the algebraic space \(\Res_{X^{\tn{ft}}\hspace{-0.1em}/S^{\tn{ft}}}\) defined by Artin in \cite{art:74b}. Let \(e\co R\ra S\) be the minimal versal base for \(\Res_{X/S}\). Then \(e\) is a finite map from \(R\) onto the Artin component \(A\) in \(S\); \cite[Thm. 3]{art:74b}. This generalises Brieskorn's (analytic) result for RDPs (then \(A=S\)). Brieskorn's use of simple Lie algebras also gave the covering with the corresponding Weyl group as Galois group; \cite{bri:70}. Wahl in \cite[Thm. 1]{wah:79} showed that if \(W_{i}\) is the Weyl group corresponding to the \(i\)-th RDP on the RDP-resolution \(\tilde{X}_{0}^{\tn{c}}\) of \(X_{0}\) then \(R\ra A\) is Galois with \(\prod W_{i}\) as group. The cruical new ingredient was that \(\df{}{\tilde{X}_{0}^{\hspace{-0.06em}\tn{c}}\hspace{-0.06em}/X_{0}}\ra \df{}{X_{0}}\) is injective. This was proved by Lipman in \cite{lip:79} (with a formulation as in Corollary \ref{cor.Lip}). 

The functor \(\Res_{X/S}\) is related to our \(\df{}{\tilde{X}_{0}/X_{0}}\) as follows. Let \((f\co Y\ra X_{R})\) be a minimal versal element in \(\Res_{X/S}(R/S)\). Then \(f\) is proper with the minimal resolution as closed fibre. Since \(\mr{R}^{2}\hspace{-0.1em}(f_0)_{*}(-)=0\) and \(\mr{R}^{1}\hspace{-0.1em}(f_0)_{*}\Q_{Y_0}=0\), Corollary \ref{cor.extbc} and Nakayama's lemma implies that \(\mr{R}^{1}\hspace{-0.18em}f_{*}\Q_{Y}=0\). Hence \(f\) gives a versal element in \(\df{}{\tilde{X}_{0}/X_{0}}(R)\) by the proof of \cite[3.3]{art:74b} (without restricting to artin rings) in all characteristics. By \cite[4.6]{art:74b} it is minimal versal if \(X_{0}\) is equivariant (e.g. if \(\Char k=0\)); cf. \cite{wah:75}.
\end{rem}
%%%%%%%%%%%%% REMARK %%%%%%%%%%%%%%
\begin{rem}
The commutative diagram in Theorem \ref{thm.square} implies that there is a commutative diagram 
\begin{equation}\label{eq.mainsq2}
\xymatrix@C+6pt@R-6pt@H-6pt{
\hspace{-0.0em}R(Y_{0},\ms{M}_{0}) \ar[r]^(0.54){b}\ar@<-0.0em>[d]_(0.44){a} & R{(Y_{0})}\hspace{-0.1em}\ar@<0.12em>[d]^(0.42){d}
\\
\hspace{-0.4em}R{(X_{0},M_{0})} \ar[r]^(0.52){c} & R{(X_{0})}\hspace{-0.18em}
}
\end{equation}
of corresponding minimal versal base spaces.  
It may be interesting to study the components of \(R(X_{0},M_{0})\). For instance, the components in \(R(Y_{0},\ms{M}_{0})\) are components in \(R(X_{0},M_{0})\) as will be shown elsewhere. One can also study the components of \(R(X_{0})\) in terms of the deformation theory of pairs for various \(M_{0}\). Since \(b\) is smooth the components of \(R(Y_{0},\ms{M}_{0})\) correspond to the components of \(R(Y_{0})\) and hence (e.g. by the discussion in Remark \ref{rem.A} and similar results for partial resolutions) to components of \(R(X_{0})\).  
Note that by Theorem \ref{thm.MW} any partial resolution dominated by the minimal resolution is obtained for some \(M_{0}\). Since \(a\) always is an embedding, Brieskorn's covering phenomena will reemerge for a restriction of the map \(c\).
\end{rem}
%%%%%%%%%%%%%%% EXAMPLE %%%%%%%%%%%%%%%%%%%%%%%%%%%%%%
\begin{ex}\label{ex.reg}
Let \(\BB{A}^{2}\) be the henselisation of \(\BB{A}^{2}_{\BB{C}}\) at the origin and \(q\co \BB{A}^{2}\ra X=\BB{A}^{2}/G\) the quotient map for a finite subgroup \(G\) of \(\Sl_{2}(\BB{C})\) so that in particular \(X\) is an RDP\textup{.} Put \(M^{\tn{reg}}=q_{*}\Q_{\hspace{-0.16em}\BB{A}^{2}}\). Then \(M^{\tn{reg}}\) corresponds to the regular \(G\)-representation (i.e.\ \(M^{\tn{reg}}\cong(q_{*}\Q_{\hspace{-0.16em}\BB{A}^{2}}[G])^{G}\) where \(G\) acts on the coefficients as well). Since all indecomposable reflexive \(\Q_{X}\)-modules are direct summands of \(M^{\tn{reg}}\), the blowing-up of \(X\) in \(M^{\tn{reg}}\) is the minimal resolution \(\pi\co \tilde{X}\ra X\) by Theorem \ref{thm.MW}\textup{.} Put \(\ms{M}^{\tn{reg}}=\pi^{\tri}M^{\tn{reg}}\)\textup{.} 
By Theorem \ref{thm.square} the map \(b\co R(\tilde{X},\ms{M}^{\tn{reg}})\ra R(\tilde{X})\) is an isomorphism, and hence \(ab^{-1}\co R(\tilde{X})\ra R(X,M^{\tn{reg}})\) is a closed immersion\textup{.} It will be shown elsewhere that the image is an irreducible component. Thus \(R(X,M^{\tn{reg}})\) has a distinguished component such that the restriction of the forgetful map \(c\co R(X,M^{\tn{reg}})\ra R(X)\) is a Galois covering (from Briskorn's result) with covering group the Weyl group with Coxeter-Dynkin diagram equal to the dual graph of the exceptional divisor in the minimal resolution.
\end{ex}
%%%%%%%%%%%%% EKSEMPEL %%%%%%%%%%%
\begin{ex}[The fundamental module]\label{ex.fund}
Let \(X_{0}\) be a normal surface singularity, \(\omega_{X_{0}}\) the canonical module, and \(\fr{m}_{0}\) the maximal \(\Q_{X_{0}}\)-ideal. There are natural isomorphisms \(\xt{1}{X_{0}}{\fr{m}_{0}}{\omega_{0}}\cong \xt{2}{X_{0}}{\Q_{X_0}/\fr{m}_{0}}{\omega_{0}}\cong k\) by local duality theory. Choose a short exact sequence
\begin{equation}\label{eq.fund}
0\ra \omega_{X_{0}}\lra F_{0}\lra \fr{m}_{0}\ra 0
\end{equation}
which represents \(1\in k\). 
It follows that \(F_{0}\) is reflexive of rank \(2\); cf. \cite[5.7]{ile:12a}. Let \(f_{0}\co Y_{0}\ra X_{0}\) denote the blowing-up in \(F_{0}\). Assume \(X_{0}\) is an RDP and \(\Char(k)=0\). We claim that the minimal versal base scheme \(R(X_{0},F_{0})\) consists of two irreducible components; \(R^{0}\) and \(R^{E}\), informally defined as:
\begin{enumerate}[leftmargin=3em] 
\item[(\(R^{0}\))] Deformations of \(X_{0}\) with a section
\item[(\(R^{E}\))] Deformations of the pair \((X_{0},F_{0})\) which give a flat blowing-up 
\end{enumerate}
For the \(\mr{A}_{n}\), \(\mr{D}_{n}\) and \(\mr{E}_{n}\) the dimensions are \(\dim R^{0}=n+2\) and \(\dim R^{E}=n\).
More specifically: Note that $\xt{j}{X_{0}}{\fr{m}_{0}}{\omega_{X_{0}}}\cong \xt{j+1}{X_{0}}{\Q_{X_0}/\fr{m}_{0}}{\omega_{X_{0}}}$ is $0$ for $j\neq 1$ by local duality theory. For \((X,I)\in\df{}{(X_{0},\fr{m}_{0})}(S)\) the base change map gives a deformation of modules \(\xt{1}{X}{I}{\omega_{X/S}}\ra \xt{1}{X_{0}}{\fr{m}_{0}}{\omega_{X_{0}}}\) by Corollary \ref{cor.extbc}. Lifting the extension \eqref{eq.fund} along this map gives an \(S\)-flat \(\Q_{X}\)-module \(F\) specialising to \(F_{0}\); cf. \cite[3.1]{ish:00}. One obtains a smooth map \(\df{}{(X_{0},\fr{m}_{0})}\ra \df{}{(X_{0},F_{0})}\); see \cite[9.11]{ile:11x}. Let  \(x_{0}\) denote the closed point in \(X_{0}\). There is a functor \(\df{}{X_{0}\ni x_{0}}\) of deformations \(X\ra S\) of \(X_{0}\) with a section \(X\la S\). The kernel of the surjection \(\Q_{X}\ra \Q_{S}\) gives an element in \(\df{}{(X_{0},\fr{m}_{0})}\) and hence a map \(\df{}{X_{0}\ni x_{0}}\ra \df{}{(X_{0},\fr{m}_{0})}\). If \(X\ra R(X_{0})\) denotes 
the minimal versal family of \(\df{}{X_{0}}\) then the base change \(X^{2}\ra X\) of \(X\ra R(X_{0})\) to \(X\) with the diagonal as section is a minimal versal family for \(\df{}{X_{0}\ni x_{0}}\); \cite[6.7]{ile:14}. Then \(R^{0}\) is defined as the image of \(X\) under the composition \(\df{}{X_{0}\ni x_{0}}\ra \df{}{(X_{0},F_{0})}\). Moreover, \(R^{E}\) is defined as the image of \(ab^{-1}\co R(Y_{0})\ra R(X_{0},M_{0})\).
A proof of the claim will be published elsewhere.
\end{ex}
%%%%%%%%%%%%%%%%%%%%%%%%%%%%%%%%%%%%%%%%%%%%%%%%%%%%%%
%%%%%%%%%%%%%% SECTION %%%%%%%%%%%%%%%%%%%%%%%%%%%%%%%
%%%%%%%%%%%%%%%%%%%%%%%%%%%%%%%%%%%%%%%%%%%%%%%%%%%%%%
\section{An application to flops}\label{sec.flops}
We apply our results to describe flops contracting to cDV-points. The results generalise the conjectures stated by Curto and Morrison in \cite{cur/mor:13}.

Let \(X_{0}\) denote an RDP and assume \(\Char(k)\neq 2\). Then \(X_{0}\) is a hypersurface singularity defined by a polynomial of the form \(F=z^{2}+d(x,y)\) by \cite{art:77}. There is a non-trivial involution \(\sigma_{0}\co X_{0}\ra X_{0}\) defined by \(z\mapsto -z\).
%%%%%%%%%%%%%%% LEMMA %%%%%%%%%%%%%%%%%%%%%%%%%%%%%%
\begin{lem}\label{lem.syz}
Suppose \(M_{0}\) is a reflexive \(\Q_{X_{0}}\)-module without free summands and let \(M_{0}^{+}\) denote the syzygy module of \(M_{0}\)\textup{.}
Let \(f_{0}\co Y_{0}\ra X_{0}\) and \(f_{0}^{+}\co Y_{0}^{+}\ra X_{0}\) be the blowing-up of \(X_{0}\) in \(M_{0}\) and \(M_{0}^{+}\)\textup{,} respectively\textup{.}
\begin{enumerate}[leftmargin=2.4em, label=\textup{(\roman*)}]
\item Taking the syzygy gives a well defined map \(\delta\co \df{}{(X_{0},M_{0})}\ra \df{}{(X_{0},M^{+}_{0})}\) which is an isomorphism\textup{.}
\item There is a unique isomorphism \(\theta_{0}\co Y_{0}\ra Y_{0}^{+}\) with \(f_{0}^{+}\theta_{0}=\sigma_{0}f_{0}\)\textup{.} Moreover\textup{,} for any deformation \(f\co Y\ra X\) in \(\df{}{Y_{0}/X_{0}}(S)\) with image \((X,M)\) in \(\df{}{(X_{0},M_{0})}(S)\), there is an involution \(\sigma\) of \(X\) extending \(\sigma_{0}\) such that the blowing-up \(f^{+}\co Y^{+}\ra X\) of \(X\) in the syzygy \(M^{+}\) is isomorphic to \(\sigma f\) by a unique isomorphism \(\theta\co Y\ra Y^{+}\) which extends \(\theta_{0}\)\textup{.}
\item The composition of \(\alpha\beta^{-1}\) for \(f_{0}\) in \textup{Theorem \ref{thm.square}} with \(\delta\) and the inverse of \(\alpha\beta^{-1}\) for \(f_{0}^{+}\) is a well defined isomorphism 
\begin{equation*}
+\co \df{}{Y_{0}/X_{0}}\xra{\hspace{1em}\simeq\hspace{1em}} \df{}{Y^{+}_{0}/X_{0}}
\end{equation*}
which is independent of \(M_{0}\) within the class of reflexive \(\Q_{X_{0}}\)-modules with \(f_{0}\) as blowing-up\textup{;} cf. \textup{Theorem \ref{thm.MW} (ii)}\textup{.}
\end{enumerate}
\end{lem}
\begin{proof}
(i) 
Fix a minimal free cover \(\vare_{0}\co \Q_{X_{0}}^{\oplus n}\ra M_{0}\) and define \(M_{0}^{+}\) as \(\ker \vare_{0}\). For a deformation \((X,M)\) of \((X_{0},M_{0})\) choose a lifting \(\vare\co \Q_{X}^{\oplus n}\ra M\) of \(\vare_{0}\) and define \(M^{+}\) as \(\ker \vare\). Then \((X,M^{+})\) is a deformation of \((X_{0},M_{0}^{+})\). Another choice of lifting of \(\vare_{0}\) gives an isomorphic deformation and \(\delta\) is well defined. Since \(X_{0}\) is a hypersurface singularity and \(M_{0}\) is MCM there is an isomorphism \(\syz{2}{X_{0}}(M_{0})\cong M_{0}\) (see \cite{eis:80}) which extends to any deformation \((X,M)\).

(ii) By \cite[2.6 (ii)]{kno:87} the pullback \(\sigma_{0}^{*}M_{0}\) is isomorphic to \(M_{0}^{+}\). It follows from Proposition \ref{prop.blowup} that \(f_{0}^{+}\) is uniquely isomorphic to \(\sigma_{0} f_{0}\). The tangent space of the (unobstructed) deformation functor \(\df{}{X_{0}}\) is given by \(\Q_{X_{0}}/(F_{x},F_{y},F_{z})\). Since \(\Char(k)\neq 2\), a versal deformation may be chosen of the form \(z^{2}+D(x,y,t)\) for some variables \(t=t_{1},\dots,t_{n}\) and hence \(X\) is isomorphic to a deformation of this form, too. Then \(\sigma_{0}\) extends trivially to an involution \(\sigma\) of \(X\). Again by \cite[2.6 (ii)]{kno:87}, \(\sigma^{*}M\cong M^{+}\). Then \(f^{+}\) is isomorphic to \(\sigma f\) by Proposition \ref{prop.blowup}.

(iii) In particular, \(Y^{+}\) is \(S\)-flat and so the map \(+\) is well defined, an isomorphism since \(+^{2}\simeq \id\), and independent of the module since we may use the same involution \(\sigma\). 
\end{proof}
%%%%%%%%%%%%%%%%%%%%% DEFINISJON %%%%%%%%%%%%%%%%%%%%%
\begin{defn}\label{defn.hyper}
For a singularity \(\Spec B\), we say that \(\Spec B/(u)\) is a \emph{good hyperplane section} if \(u\) is a non-zero-divisor contained in \(\fr{m}_{B}{\smallsetminus}\fr{m}^{2}_{B}\) such that \(\Spec B/(u)\) is an isolated singularity. With \(T=\Spec k[t]^{\tn{h}}\) the associated map \(\Spec B\ra T\) defined by \(t\mapsto u\) is called the \emph{hyperplane section map}.

If \(\dim\Spec B=3\) and \(\Spec B/(u)\) is an RDP for a generic choice of \(u\in \fr{m}_{B}{\smallsetminus}\fr{m}^{2}_{B}\), \(\Spec B\) is called a cDV; cf. \cite[5.32]{kol/mor:98}.
\end{defn}
Assume \(g\co W\ra Z\) is a small partial resolution of a normal singularity, \(K_{W}\) is numerically \(g\)-trivial and that \(D\) is a \(\BB{Q}\)-Cartier divisor on \(W\) such that \(-D\) is \(g\)-ample. Then a \(D\)-flop of \(g\) is a partial resolution \(g^{+}\co W^{+}\ra Z\) such that the strict transform \(D^{+}\) of \(D\) to \(W^{+}\) is \(g^{+}\)-ample; cf. \cite[6.10]{kol/mor:98} and \cite{kol:89}. If \(\varSigma(g)\) is irreducible, then \(g^{+}\) is called a simple flop of \(g\). If \(\dim Z=3\), the length of a simple flop is defined as the length at the generic point of \(E(g)\); see \cite[16.7]{cle/kol/mor:88}. In a flop, $W$ and $W^+$ typically share many properties, e.g. the number and type of singularities \cite[2.4]{kol:89}.

Assume \(\Char (k)=0\) for the rest of the article.
We will consider the case where \(Z\) is an isolated cDV which is equivalent to \(Z\) being Gorenstein and terminal; cf. \cite[5.38]{kol/mor:98}. Moreover, \(Z\) is rational by R.\ Elkik's \cite[Th{\'e}. 2]{elk:78}; cf. \cite[5.42]{kol/mor:98}. By a theorem of Reid any crepant partial resolution \(g\co W\ra Z\) is small, any good hyperplane section \(X\subset Z\) has a normal strict transform \(Y\subset W\) and the induced map \(f\co Y\ra X\) is a partial resolution of an RDP dominated by the minimal resolution, see \cite[1.14]{rei:83}. This allows us to apply Theorem \ref{thm.square}. We show that \(g\) and its flop \(g^{+}\) is given as a blowing-up in an MCM module and in its syzygy module. In addition to existence the construction gives the flops independence of the divisor \(D\).
%%%%%%%%%%%%%%%%%%%%%%% TEOREM %%%%%%%%%%%%%%%%%%%%
\begin{thm}\label{thm.flop}
Suppose \(g\co W\ra Z\) is a small partial resolution of an isolated cDV singularity\textup{.} Let \(D\) be a Cartier divisor on \(W\) such that \(-D\) is \(g\)-ample\textup{.}  
Then\textup{:}
\begin{enumerate}
[leftmargin=2.4em, label=\textup{(\roman*)}]
\item There is a maximal Cohen-Macaulay \(\Q_{Z}\)-module \(M\) such that \(\bigwedge^{\rk M}g^{\tri}M\cong \Q_{W}(-D)\) and \(g\) is isomorphic to the blowing-up \(\Bl_{M}(Z)\ra Z\)\textup{.}
\item Let \(M^{+}\) denote the syzygy module of \(M\)\textup{.} Then
\begin{equation*}
\Bl_{M}(Z)\xra{\hspace{1.3em}} Z \xla{\hspace{0.5em}g^{+}\hspace{0.3em}} \Bl_{M^{+}}(Z)=W^{+}
\end{equation*}
gives the unique \(D\)-flop of \(g\) and 
\(\bigwedge^{\rk M^{+}}\hspace{-0.2em}(g^{+})^{\tri}M^{+}\cong \Q_{W^{+}}(D^{+})\) 
where \(D^{+}\) is the strict transform of \(D\) to \(W^{+}\)\textup{.} 
\item Given \(g\)\textup{,} the \(D\)-flop is independent of the Cartier divisor \(D\)\textup{.}
\item If the flop is simple, \(M\) can be chosen to be indecomposable and then the length of the flop equals \(\rk M\)\textup{.}
\end{enumerate}
\end{thm}
\begin{proof}
(i) Let \(f_{0}\co Y_{0}\ra X_{0}\) be the strict transform along \(g\) of a good hyperplane section of \(Z\). Then \(f_{0}\) is a partial resolution of the RDP \(X_{0}\) dominated by the minimal resolution; \cite[1.14]{rei:83}. With \(T=\Spec k[t]^{\tn{h}}\), the hyperplane section map gives \(g\) as an element in \(\df{}{Y_{0}/X_{0}}(T)\). Let \(j\co Y_{0}\ra W\) denote the closed embedding. By Proposition \ref{prop.pic} the restriction \(j^{*}\co \Pic W\ra \Pic Y_{0}\) is an isomorphism where the ample sheaves are in correspondence. In particular \(j^{*}\Q_{W}(-D)\) is ample, isomorphic to \(\ch_{1}(f_{0}^{\tri}M_{0})\) for a reflexive \(\Q_{X_{0}}\)-module \(M_{0}\) and \(f_{0}\) is the blowing-up of \(X_{0}\) in \(M_{0}\) by Theorem \ref{thm.MW}. We may assume \(M_{0}\) is without free summands. By Theorem \ref{thm.square} and Proposition \ref{prop.blowingup} the image of \(g\) in \(\df{}{(X_{0},M_{0})}(T)\) gives a pair \((Z,M)\) such that \(g\) is the blowing-up of \(Z\) in \(M\). Note that \(\depth M=\depth \Q_{T}+\depth M_{0}=1+2\) so \(M\) is MCM. By Lemma \ref{lem.U} and Corollary \ref{cor.blowup}, \(j^{*}g^{\tri}M\cong f_{0}^{\tri}M_{0}\) and hence \(\Q_{W}(-D)\cong\ch_{1}(g^{\tri}M)\) by Proposition \ref{prop.pic} (iv).

(ii) 
By Proposition \ref{prop.pic} we may assume that \(-D\) is an effective divisor intersecting the \(g\)-exceptional locus transversally, hitting all components. 
Put \(\bar{D}=g_{*}(D)\); a Weil divisor. By Lemma \ref{lem.syz} there is an involution \(\sigma\) on \(Z\) and \(\sigma g\co W\ra Z\) is isomorphic to \(g^{+}\). In particular \(D^{+}\) is Cartier.  
There is a degree \(2\) covering \(Z\ra P\) where \(P\) is regular and \(\sigma \) is the covering involution. Since \(\sigma_{*}(-\bar{D})-\bar{D}\) is \(\sigma\)-invariant, it is the pullback of a (principal) Cartier divisor on \(P\); cf. \cite[2.3]{kol:89}. By Lemma \ref{lem.MJ} there is a short exact sequence (\(r=\rk M\)):
\begin{equation}\label{eq.weil}
0\ra \Q_{Z}^{\oplus r-1}\xra{\hspace{0.4em}s\hspace{0.4em}} M\lra g_{*}\Q_{W}(-D) \ra 0 
\end{equation} 
By \cite[2.6 (ii)]{kno:87}, \(\sigma^{*}M\cong M^{+}\). If \(i\co U \hra Z\) denotes the inclusion of the regular locus, the restriction map \(g_{*}\Q_{W}(-D)\ra i_{*}i^{*}g_{*}\Q_{W}(-D)\) is an isomorphism since \eqref{eq.weil} implies \(\depth g_{*}\Q_{W}(-D)\geq 2\). It follows that \(\sigma_{*}g_{*}\Q_{W}(-D)\cong g^{+}_{*}\Q_{W^{+}}(D^{+})\) since \(\sigma_{*}(-\bar{D})\sim \bar{D}=g_{*}^{+}(D^{+})\). Then \(\sigma_{*}\) (\(=\sigma^{*}\)) applied to \eqref{eq.weil} gives  
the short exact sequence 
\begin{equation}\label{eq.weil+}
0\ra \Q_{Z}^{\oplus r-1}\xra{\hspace{0.4em}\sigma^{*}s\hspace{0.4em}} M^{+}\lra g^{+}_{*}\Q_{W^{+}}(D^{+}) \ra 0 
\end{equation}
and \(\bigwedge^{\hspace{-0.12em}r}(g^{+})^{\tri}M^{+}\) is isomorphic to \(\Q_{W^{+}}(D^{+})\) by restricting to \(U\) and extending to \(W^{+}\); cf. \eqref{eq.chern}. In particular, \(D^{+}\) is ample by Proposition \ref{prop.pic} and Theorem \ref{thm.MW} as in the proof of (i). If \(g^{\sharp}\co W^{\sharp}\ra Z\) is another \(D\)-flop of \(g\) and \(D^{\sharp}\) the strict transform of \(D\), then \(g^{\sharp}_{*}(\Q_{W^{\sharp}}(D^{\sharp}))\cong g_{*}\Q_{W}(-D)\cong \llb M^{+}\rrb\) (Lemma \ref{lem.MJ}) and \(g^{\sharp}\cong g^{+}\) by \cite[6.2]{kol/mor:98}. 

(iii) Let \(D'\) be a Cartier divisor on \(W\) such that \(-D'\) is ample. By the above construction, \(g\) is given by blowing up \(Z\) in a maximal Cohen-Macaulay module \(M'\). The \(D'\)-flop which is given by blowing up \(Z\) in the syzygy \((M')^{+}\) is a deformation in \(\df{}{Y_{0}^{+}/X_{0}}(T)\) equal to \(f^{+}\) by Lemma \ref{lem.syz} (iii).

(iv) Since \(E(f_{0})\) is irreducible, we can by Theorem \ref{thm.MW} assume that \(M_{0}\) is indecomposable and hence that the rank of \(M_{0}\) is the intersection number \(\mr{c}_{1}(\ms{M}_{0}).E(f_{0})\) which equals the length of the scheme \(E(f_{0})\) at its generic point. 
By Proposition \ref{prop.pic} (iv) this is also the length of \(E(g)\) at its generic point which is the length of the flop. 
\end{proof}
%%%%%%%%%%%%%%%%%%%%% REMARK %%%%%%%%%%%%%%%%%%%%%
\begin{rem}
The flop's independence of the divisor \(D\) (even though the contraction \(g\) is not necessarily extremal) is known; e.g. \cite[below Def. 3]{kol:90}. 
\end{rem}
%%%%%%%%%%%%%%%%%%%%% REMARK %%%%%%%%%%%%%%%%%%%%%
\begin{rem}\label{rem.W}
Theorem \ref{thm.flop} is directly motivated by Curto and Morrison's conjectures \cite[Conj. 1-3]{cur/mor:13} about simple flops described in terms of matrix factorisations which they hoped would enable more explicit versions of the Bridgeland-Chen theorem and its applications. They also noted that Van den Bergh's approach in \cite{vdber:04} seemed closely related to their own. 
Assume $g\co W\ra Z$ is a projective map with $Z$ a singularity of arbitrary dimension, $g$ has at most $1$-dimensional fibres, $\mr{R}^1g_*\Q_W=0$, and $E(g)_{\tn{red}}=\cup E_i$. Van den Bergh constructs a  projective generator $\ms{P}=\Q_W{\oplus}\ms{M}$ for the category $\hspace{-0.2em}{}^{-1}\mr{Per}(W/Z)$ such that $\ms{Q}=\Q_W{\oplus}\ms{M}^{\vee}\hspace{-0.15em}$ is a projective generator for $\hspace{-0.2em}{}^{0}\mr{Per}(W/Z)$; \cite[3.2.7]{vdber:04}. Moreover, $\ms{M}=\bigoplus\ms{M}_i$ for locally free sheaves $\ms{M}_i$ that are generalisations of the strict transform of Wunram modules with $\mr{c}_1(\ms{M}_i).E_j=\delta_{ij}$; \cite[3.5.5]{vdber:04}. In particular $\ms{M}=g^{\tri}M$ for $g$ and $M$ as in Theorem \ref{thm.flop}. With further conditions (normality, $g$ birational, $\codim \varSigma(g)\geq 2$ and $Z$ a canonical hypersurface singularity of multiplicity $2$) there exists a flop $g^+=\sigma g$ for an involution $\sigma$ by \cite[2.2-3]{kol:89}. Put $M=g_*\ms{M}$. Van den Bergh shows that the corresponding $\ms{M}^+$ for $g^+\hspace{-0.1em}\co W^+\ra Z$ satisfies $g^+_*\ms{M}^+\cong M^{\vee}$; \cite[4.3.1]{vdber:04}. Put $\ms{P}^+=\Q_{W^{\hspace{-0.05em}+\hspace{-0.05em}}}{\oplus}\ms{M}^+$ and $\ms{Q}^+=\Q_{W^{\hspace{-0.05em}+\hspace{-0.05em}}}{\oplus}(\ms{M}^+)^{\vee}$. His main result \cite[4.4.2]{vdber:04} implies that $W$ and $W^+$ both are derived equivalent with $\nd{}{W}{\ms{P}}\cong\nd{}{Z}{\Q_Z{\oplus} M}\cong \nd{}{W^{\hspace{-0.05em}+\hspace{-0.05em}}}{\ms{Q}^+}$ such that $\hspace{-0.2em}{}^{-1}\mr{Per}(W/Z)\simeq \cat{Coh}(\nd{}{W}{\ms{P}})\simeq \hspace{-0.2em}{}^{0}\mr{Per}(W^+\hspace{-0.2em}/Z)$. 

We note that since $M$ is MCM by \cite[3.2.9]{vdber:04}, there is an isomorphism of $M^{\vee}\cong\sigma^*M$ with the syzygy module $\Syz M$ by \cite[2.6 (ii)]{kno:87}. This implies that $g^+_*\ms{P}^+$ is isomorphic to $\Q_Z{\oplus}\Syz M$. With $g$ and $M$ as in Theorem \ref{thm.flop} we get that $W^+\cong \Bl_{g_*\hspace{-0.08em}\ms{Q}}(Z)$.

Wemyss and collaborators have developed these ideas in several directions. Put $\vL=\nd{}{W}{\ms{Q}}\cong \nd{}{W}{\ms{P}}^{\tn{op}}$.
While Van den Bergh has no construction of the flop maps, Karmazyn \cite[5.2.4]{kar:17} reconstructs $g$ (in a more general situation) as a quiver GIT moduli space $\ms{M}_{\rk,\vartheta}(\vL)\ra Z$ where the ranks of the indecomposable summands in $\ms{P}$ determine the dimension vector $\rk$ and the stability condition $\vartheta$; \cite[5.1.2]{kar:17}. 
This contrasts with our direct, geometric construction in Theorem \ref{thm.flop} by blowing up in a MCM module and (for the flop) in its syzygy and it would be interesting to know how the two approaches are related. 
 
Assume $Z$ is $3$-dimensional and Gorenstein, $W$ is Gorenstein with terminal singularities, $g$ is birational, $\dim E(g)=1$, and $\mr{R}^1\hspace{-0.1em}g_*\Q_W=0$; \cite[2.9]{wem:18}. 
Suppose a subset ${\cup_{i\in I}E_i}$ is contracted by a small birational map $g_I\co W\ra W_{\hspace{-0.1em}I}$ with $h\co W_I\ra Z$ and $g=hg_I^+$ and with flop $g_I^+\co W^+\ra W_{\hspace{-0.1em}I}$. Put $g^+=hg_I^+$. 
Wemyss defines mutation operators $\nu_I$ and $\mu_I$ \cite[2.18]{wem:18} such that $g^+_*\ms{Q}^+\cong \nu_I(\Q_Z{\oplus} M^{\vee})$; \cite[4.2]{wem:18}.
The translation of flop to mutation of the module on $Z$ allows better control, e.g. of possible new flops and relations to the chamber structure in the quiver GIT moduli spaces, as demonstrated in \cite{wem:18}. We note that in the case $h=\id$ (i.e. all curves are flopped), $\nu(\Q_Z{\oplus} M^{\vee})=\Q_Z{\oplus} (\Syz M)^{\vee}$ and $\mu(\Q_Z{\oplus} M)=\Q_Z{\oplus} \Syz M$ by definition, which ties our construction of the flop to Wemyss' \cite[4.19]{wem:18}. With assumptions as in Theorem \ref{thm.flop}, $W^+\cong\Bl_{\mu M}(Z)$. One may ask if this equation generalises.
\end{rem}

We now consider the relative case. 
First some notation needed in the statement of Theorem \ref{thm.CM}.
Let \(T=\Spec k[t]^{\tn{h}}\) and \(T_{S}=T{\times}^{\tn{h}} S\) for \(S=\Spec A\) and \(A\) any henselian local \(k\)-algebra.
Let \(\Spec B\ra S\) be a local family of singularities, with central fibre \(\Spec B_{0}\). If \(u\in\fr{m}_{B}\) maps to \(u_{0}\in B_{0}\) then \(u_{0}\) is a non-zero-divisor if and only if \(u\) is a non-zero-divisor and \(\Spec B/(u)\) is \(S\)-flat; cf.\ \cite[19.2.4]{EGAIV4}. Moreover, 
\(t\mapsto u\) defines a flat 
map \(\Spec B\ra T_{S}\) which extends \(\Spec B_{0}\ra T\) defined by \(u_{0}\).

Suppose \(g\co W\ra Z\) is a local family over \(S\) where the central fibre \(g_{0}\co W_{0}\ra Z_{0}\) is a small partial resolution of a cDV singularity. Let \(f_{0}\co Y_{0}\ra X_{0}\) be the strict transform along \(g_{0}\) of any good hyperplane section \(X_{0}\) of \(Z_{0}\). 
Then \(f_{0}\) is a partial resolution of an RDP dominated by the minimal resolution; \cite[1.14]{rei:83}. 
By Corollary \ref{cor.square} there is a versal family \(Y\xra{\hspace{0.3em}f\hspace{0.3em}} X\ra R\) for \(\df{}{Y_{0}/X_{0}}\).
By Theorem \ref{thm.MW} there exists a reflexive \(\Q_{X_{0}}\)-module such that blowing up \(X_{0}\) in it gives \(f_{0}\). 
With these notions fixed we have:  
%%%%%%%%%%%%%%% THEOREM %%%%%%%%%%%%%%%%%%%%%%%%%%%%%%
\begin{thm}\label{thm.CM}\hspace{-0.5em}
For every reflexive \(\Q_{X_{0}}\)-module \(M_{0}\) such that \(f_{0}\) is given by blowing up \(X_{0}\) in \(M_{0}\)\textup{,} there is a deformation \((X,M)\) in \(\df{}{(X_{0},M_{0})}(R)\) with the following properties\textup{:}
\begin{enumerate}[leftmargin=2.4em, label=\textup{(\roman*)}]
\item 
Let \(Z\ra T_{S}\) be an extension of the hyperplane section map \(Z_{0}\ra T\)\textup{.} Then there is a map \(h\co T_{S}\ra R\) such that \(g\co W\ra Z\) is the base change of \(f\) along \(h\)\textup{.} 
\item 
Let \(N\) be the base change of \(M\) along \(h\)\textup{.} 
Then \(g\) is the blowing-up of \(Z\) in \(N\)\textup{. }
\item 
Let \(N^{+}\hspace{-0.3em}\)  
denote the syzygy module of \(N\)\textup{.} 
Blowing up \(Z\) in \(N^{+}\hspace{-0.3em}\) gives a local family \(g^{+}\co W^{+}\hspace{-0.3em}\ra Z\) with central fibre \(g_{0}^{+}\) which is the unique flop of \(g_{0}\)\textup{.}
\end{enumerate}
\end{thm}
\begin{proof}
Let \(\ms{M}_{0}\) denote the strict transform \(f_{0}^{\tri}M_{0}\).
Let \((Y/X,\ms{M}/M)\) be the versal element in \(\df{}{(Y_{0}/X_{0},\ms{M}_{0}/M_{0})}(R)\) corresponding to \(f\) by Theorem \ref{thm.square}. Then \((X,M)\in \df{}{(X_{0},M_{0})}(R)\).

(i) Note that \(W\ra Z\ra T_{S}\) is an element in \(\df{}{Y_{0}/X_{0}}(T_{S})\). Use versality of \(f\). 

(ii) 
Blowing up \(X\) in \(M\) gives \(f\co Y\ra X\) back and the strict transform of \(M\) is \(\ms{M}\); see Proposition \ref{prop.blowingup}. Then \(g\) is the blowing-up in \(N\) since blowing-up commutes with base change by Lemma \ref{lem.U} and Corollary \ref{cor.blowup}. 

(iii) Let \(M^{+}\) denote the syzygy module of \(M\). Then the central fibre of \(M^{+}\) equals the syzygy \(M_{0}^{+}\) of \(M_{0}\) and the blowing-up of \(X_{0}\) in \(M_{0}^{+}\) gives by Theorem \ref{thm.MW} a partial resolution \(f_{0}^{+}\co Y_{0}^{+}\ra X_{0}\). Put \(\ms{M}_{0}^{+}=(f_{0}^{+})^{\tri}M_{0}^{+}\). By Lemma \ref{lem.syz}, Theorem \ref{thm.square} and Proposition \ref{prop.blowingup} blowing up \(X\) in \(M^{+}\) gives a versal element \(f^{+}\co Y^{+}\ra X\) in \(\df{}{Y_{0}^{+}/ X_{0}}(R)\). Let \(\sigma\) be the involution of \(X\) extending \(\sigma_{0}\) given in Lemma \ref{lem.syz}. Taking syzygies and blowing up commutes with base change (Lemma \ref{lem.U} and Corollary \ref{cor.blowup}). The pullback of \(f^{+}\) by \(h\) gives a map \(g^{+}\co W^{+}\ra Z\) such that its central fibre \(g_{0}^{+}\) is a small partial resolution. The involution \(\sigma\) pulls back to an involution 
\(h^{*}\sigma\) of \(Z\) and \((h^{*}\sigma) g = g^{+}\).
In particular, \(g_{0}^{+}\) is the unique flop of \(g_{0}\); see Theorem \ref{thm.flop}.     
\end{proof}
%%%%%%%%% REMARK %%%%%%%%%
\begin{rem}\label{rem.CM}
Our results imply the three conjectures stated by Curto and Morrison in \cite{cur/mor:13}.
Conjecture 1 states that every simple flop (i.e. of a simple, small resolution) of length \(l\) is given by blowing up two maximal Cohen-Macaulay modules of rank \(l\). In Conjecture 2 it is stated that the two modules are syzygy modules of each other. This is contained in Theorem \ref{thm.flop}.  
Conjecture 3 says that for a simple partial resolution \(Y_{0}\ra X_{0}\) of an RDP and \(Y\ra X\ra R\) a versal element in \(\df{}{Y_{0}/X_{0}}\), there is an \(\Q_{X}\)-module \(M\) such that the pair \((X,M)\) is in \(\df{}{(X_{0},M_{0})}(R)\) as in Theorem \ref{thm.CM}.  
Moreover, \(Y\ra X\) is the blowing-up of \(X\) in \(M\) and the blowing-up of \(X\) in \(M^{+}\) gives a versal family in \(\df{}{Y_{0}^{+}/X_{0}}\). This is not contained in Theorem \ref{thm.CM}, but follows directly from Theorem \ref{thm.square}, Proposition \ref{prop.blowingup} and Lemma \ref{lem.syz}.

The conjectures also contain some statements about matrix factorisations. Recall that any MCM module on a hypersurface singularity \(\Spec Q/(F)\) is obtained as \(\coker \Phi\) for some pair \((\Phi,\Psi)\) of endomorphisms of a free finite rank module on the non-singular ambient space \(\Spec Q\) where \(\Phi\Psi=F\!{\bdot}\hspace{-0.10em}{\id}=\Psi\Phi\); see \cite{eis:80}. The family of deformations \(X\) in Theorem \ref{thm.CM} can be written as \(\Spec Q/(F)\) for a hypersurface polynomial of the form \(F=z^{2}+G(x,y,t)\) where \(t=t_{1},\dots,t_{n}\) since it is given as a base change of the versal family of an RDP. Conjecture 3 says that there is a matrix factorisation \((\Phi,\Psi)\) of \(F\) representing \(M\) with 
\begin{equation}
\Phi=zI_{2l}+\Theta\qquad\tn{and}\qquad \Psi=zI_{2l}-\Theta
\end{equation}
where \(\Theta\) is a \((2l{\times}2l)\)-matrix with entries from \(k[x,y,t]^{\tn{h}}\), \(l=\rk M\) and \((\Theta,\Theta)\) gives a matrix factorisation of \(-G\). This is however true for any hypersurface \(z^{2}+G(\tn{some other variables})\) as was observed by H.\ Kn{\"o}rrer; see the proof of \cite[2.6 (ii)]{kno:87}. Indeed, put \(P=k[x,y,t]^{\tn{h}}\) and \(\BB{A}=\Spec P\). Then \(M\) is free as \(\Q_{\BB{A}}\)-module of rank \(2l\). Multiplication on \(M\) with \(z\) defines an \(\Q_{\BB{A}}\)-linear map \(\Theta\) with \(\Theta^{2}=-G\!{\bdot}\hspace{-0.10em}{\id}\) and \((\Phi,\Psi)\) is as required. Conjecture 2 contains a very similar statement.
\end{rem} 
\begin{rem}\label{rem.uni}
We believe Theorem \ref{thm.CM} also gives (and clarifies) `the universal flop' in Remark (2) on p.\ 13 in \cite{cur/mor:13} and in \cite[Thm. 5.1]{cur/mor:13}.
With notation as in Theorem \ref{thm.CM}, let  
$f^+\hspace{-0.1em}\co Y^+\ra X$ denote the blowing-up of $X$ in the syzygy module $M^+$ and let $f_0^+\hspace{-0.1em}\co Y_0^+\ra X_0$ be the closed fibre.  
If $M^+$ is isomorphic to $M$, $f$ is isomorphic to $f^+$ and no pullbacks of $(f,f^+)$ can be flops. 
But if $(g_0\co W_0\ra Z_0, g_0^+\hspace{-0.1em}\co W_0^+\ra Z_0)$ is a flop over a cDV and $f_0\co Y_0\ra X_0$ is the strict transform along $g_0$ of a good hyperplane section $X_0$ of $Z_0$ then Theorem \ref{thm.CM} gives a map $Z_0\ra X$ such that the flop is the pullback of $(f,f^+)$ along $Z_0\ra X$. In this sense all local $3$-dimensional flops of terminal index $1$-singularities with a given type of strict transform $f_0$ of a good hyperplane section are pullbacks from the same pair of maps $(f,f^+)$. But $(f,f^+)$ is not a family of flops parametrized by $R$ in the usual sense (e.g. as $(g,g^+)/S$ in Theorem \ref{thm.CM}). Note that the map $Z_0\ra X$, as for versal families, is not unique. Note also that for a given flop there will be many different good hyperplane sections. I.e. the same flop is the pullback from many different `universal flops'. As an example consider Reid's family of flops $Z_0\co x^{2}+yz-t^{2n}$ which are $\mr{cA}_{1}$, but also gives $X_0\cong\mr{A}_{2r-1}$ for $r< n$ by $x=t^r$ and $X_0\cong\mr{A}_{2rn-1}$ by $t=x^r$.
\end{rem}
\begin{rem}\label{rem.VdB}
Our results generalise Curto and Morrison's Conjecture 1 and 2 to local families of possibly non-simple small partial resolutions. Theorem \ref{thm.CM} also shows that a local family of flops is the pullback of a pair $(f,f^+)$ as in Remark \ref{rem.uni}.
This can be turned around to construct some contractions with fibre dimension $1$ and their flops in higher dimensions. Suppose, for a normal singularity $Z$ of dimension $n\,{>}\,3$, there is a sequence of $n-3$ hyperplane sections producing a cDV. This gives a flat family $Z\ra (\BB{A}^{n-3})^{\tn{h}}$. If the strict transform $g_0$ of the hyperplane sections is a small partial resolution, Theorem \ref{thm.CM} would apply to produce $g$ and its flop $g^+$ by blowing up an MCM and its syzygy on $Z$. Even without any $g$, but with an MCM $\Q_Z$-module $M$, $n-2$ hyperplane sections make the pair $(Z,M)$ to a deformation of an RDP with a reflexive module $(X_0,M_0)$. 
Let $f_0\co Y_0\ra X_0$ be the blowing-up of $X_0$ in $M_0$ and $\ms{M}_0=f_0^{\tri}M_0$. With notation as in \eqref{eq.mainsq2}, if the induced map $(\BB{A}^{n-2})^{\tn{h}}\ra R(X_0,M_0)$ factors through the image of $R(Y_0,\ms{M}_0)$ under the closed immersion $a$, then a small partial resolution $g\co W\ra Z$ is obtained by pullback of the versal family in $\df{}{(Y_0/X_0,\ms{M}_0/M_0)}$ and $g$ is also the blowing-up of $Z$ in $M$. Moreover, the pullback of the versal family in $\df{}{(Y_0^+/X_0,\ms{M}_0^+/M_0^+)}$ along the same map, see Lemma \ref{lem.syz}, gives the flop $g^+\hspace{-0.1em}\co W^+\ra Z$ which also is the blowing-up of $Z$ in the syzygy $M^+$.

In this section our aim has been to prove (and generalise) the Curto-Morrison conjectures. We appreciate that the efforts of Van den Bergh and Wemyss are concerned with more general contractions, but many of their statements require a Gorenstein condition. One may ask to what extent our Theorem \ref{thm.square}, which is working for all rational surface singularities, can be applied to more general CM singularities. The blowing-up in a sheaf is a very general technique.
It seems that at least some of the more general contractions (e.g. as in \cite[4.4.2]{vdber:04}) are obtained as blowing-ups, e.g. in $g_*\ms{P}$, cf. Remark \ref{rem.W}. Since the blowing-up has a universal property this could be useful.
\end{rem} 
\providecommand{\bysame}{\leavevmode\hbox to3em{\hrulefill}\thinspace}
\providecommand{\MR}{\relax\ifhmode\unskip\space\fi MR }
% \MRhref is called by the amsart/book/proc definition of \MR.
\providecommand{\MRhref}[2]{%
  \href{http://www.ams.org/mathscinet-getitem?mr=#1}{#2}
}
\providecommand{\href}[2]{#2}

%\bibliography{/Users/ile/Documents/Arbeid/Aktuelle/Generiske/Matematikk10}
\end{document}